\documentclass[11pt,a4paper]{amsart}
\usepackage[utf8]{inputenc} 
\usepackage{tikz}
\usepackage{tikz-cd}
\usepackage{geometry}
\usetikzlibrary{decorations.markings,decorations.pathmorphing,cd}
\usetikzlibrary{arrows}
\usetikzlibrary{calc}
\usepackage{pgfplots}
\pgfplotsset{every axis/.append style={
                    axis x line=middle,    
                    axis y line=middle,    
                    axis line style={-,color=blue}, 
                    xlabel={$x$},          
                    ylabel={$y$},          
            }}

\usepackage{amsmath}
\usepackage{amsthm}
\usepackage{amsfonts}
\usepackage{amssymb}
\usepackage{amsrefs}
\usepackage{enumerate}
\usepackage{enumitem}
\usepackage{graphicx}
\usepackage{hyperref}
\usepackage{nicefrac}
\usepackage{cancel}
\usepackage{colortbl}

\definecolor{grisclarito}{rgb}{.9,.9,.9}

\DeclareMathOperator{\orb}{orb}

\DeclareMathOperator{\Hom}{Hom}

\DeclareMathOperator{\Deck}{Deck}
\DeclareMathOperator{\im}{Im}

\def\cP{{\mathcal P}}

\def\cC{\mathcal{C}}

\def\A{\mathbb{A}}
\def\ZZ{\mathbb{Z}}
\def\TT{\mathbb{T}}
\def\BB{\mathbb{B}}

\def\CC{\mathbb{C}}

\def\QQ{\mathbb{Q}}
\def\FF{\mathbb{F}}
\def\QQ{\mathbb{Q}}

\def\RR{\mathbb{R}}

\def\PP{\mathbb{P}}
\def\GG{\mathbb{G}}

\def\j{s}

\def\k{k}

\newcommand{\one}{\mathbf{1}}
\def\cco{{\sc{cco}}}
\def\NINF{{\sc{ninf}}}

\def\rightmap#1{\smash{\mathop{\rightarrow}\limits^{#1}}}

\newtheorem{thm}{Theorem}[section]  
\newtheorem{main-thm}{Theorem}  
\newtheorem{prop}{Proposition}[section]%
\newtheorem{main-conj}{Conjecture}[section]%
\newtheorem{cor}{Corollary}[section]
\newtheorem{lemma}{Lemma}[section]

\theoremstyle{remark}
\newtheorem{rem}{Remark}[section]

\theoremstyle{definition}
\newtheorem{dfn}{Definition}[section]
\newtheorem{exam}{Example}[section]
\newtheorem{lst}{List}[section]
\newtheorem{problem}{Problem}[section]

\numberwithin{figure}{section}

\makeatletter
\let\c@lemma\c@thm
\let\c@prop\c@thm
\let\c@propdef\c@thm
\let\c@proper\c@thm
\let\c@problem\c@thm
\let\c@conj\c@thm
\let\c@cor\c@thm
\let\c@rem\c@thm
\let\c@dfn\c@thm
\let\c@notation\c@thm
\let\c@exam\c@thm
\makeatother

\title
{Geometric realizability of epimorphisms to curve orbifold groups}

\author[Jos\'e I. Cogolludo-Agust{\'i}n and Eva Elduque]{J.I.~Cogolludo-Agust{\'i}n and Eva Elduque}
\address{Departamento de Matem\'aticas, IUMA\\
Universidad de Zaragoza \\
C.~Pedro Cerbuna 12 \\
50009 Zaragoza, Spain.} 
\email{jicogo@unizar.es} 

\address{Departamento de Matem\'aticas, ICMAT\\ 
Universidad Aut\'onoma de Madrid \\
28049 Madrid, Spain.}
\email{eva.elduque@uam.es}

\begin{document}

\thanks{The authors are partially supported by PID2020-114750GB-C31, funded by 
MCIN/AEI/10.13039/501100011033. The first author is also partially funded by the Departamento de Ciencia, 
Universidad y Sociedad del Conocimiento of the Gobierno de Arag\'on 
(Grupo de referencia E22\_20R ``\'Algebra y Geometr\'{\i}a''). 
The second author is partially supported by the Ram\'on y Cajal Grant RYC2021-031526-I funded by
MCIN/AEI/10.13039/501100011033 and by the European Union NextGenerationEU/PRTR
}

\subjclass[2020]{Primary 32S25, 32S20, 14F35, 14C21; Secondary 32S50, 20H10}

\maketitle

\begin{abstract}
\begin{sloppypar}
	Given a connected dense Zariski open set of a compact K\"ahler manifold~$U$, we address the general problem of the existence of surjective holomorphic maps ${F:U\to C}$ to smooth complex quasi-projective curves from properties of
	$\pi_1(U)$. It is known that, if such $F$ exists, then there exists a finitely generated normal subgroup $K\trianglelefteq\pi_1(U)$ such that $\pi_1(U)/K$ is isomorphic to a curve orbifold group $G$ (i.e. the orbifold fundamental group of a smooth complex quasi-projective curve endowed with an orbifold structure). In this paper, we address the converse of that statement in the case where the orbifold Euler characteristic of $G$ is negative, finding a (unique) surjective holomorphic map $F:U\to C$ which realizes the quotient
	$\pi_1(U)\twoheadrightarrow \pi_1(U)/K\cong G$ at the level of (orbifold) fundamental groups.
	We also prove that our theorem is sharp, meaning that the result does not hold for any curve orbifold group
	with non-negative orbifold Euler characteristic.
	Furthermore, we apply our main theorem to address Serre's question of which orbifold fundamental groups of smooth
	quasi-projective curves can be realized as fundamental groups of complements of curves in~$\PP^2$.
\end{sloppypar}
\end{abstract}


\section{Introduction}

Unless otherwise stated, every variety appearing in this paper is over $\CC$, every complex analytic variety is assumed to be connected and reduced, and every Zariski open set inside of a complex analytic variety is assumed to be non-empty.

Consider the following
well-known problem, which we will refer to as the Geometric Realizability Problem.

\begin{problem}[Geometric Realizability Problem]
\label{pbm:geometricmorphism}
Let $U$ be a Zariski open set of a compact K\"ahler manifold, let $u\in U$, and let $\psi:\pi_1(U,u)\to G$
be an epimorphism  with finitely generated kernel, where $G$ is a group which is isomorphic to the fundamental
group of a smooth quasi-projective curve.

Determine under what conditions
there exists an admissible map (see Definition~\ref{dfn:admissible} below)
$F:U\to C$ to a smooth quasi-projective curve $C$ realizing $\psi$, that is, such that $\psi$ and
$$F_*:\pi_1(U,u)\to \pi_1(C, F(u))$$ coincide up to isomorphism in the target.
\end{problem}

Related questions regarding the existence of maps from smooth varieties onto curves from
properties of their fundamental groups have been considered by other authors
(cf.~\cite{Arapura-fundamentalgroups,Arapura-geometry,ACM-characteristic,Bauer-irrational,Catanese-Fibred,Green-Lazarsfled-higher,Gromov-fundamental,Jost-Yau-Harmonic,Hillman-Complex})

The purpose of this paper is to study a generalization of the Geometric Realizability Problem
(see Problem~\ref{pbm:orbifoldgeometricmorphism} below) to a broader class of groups.
Our main motivation is the study of quasi-projective groups, that is, fundamental groups of smooth
quasi-projective varieties. Indeed, smooth quasi-projective varieties have smooth projective compactifications,
and thus they are Zariski open sets inside of compact K\"ahler manifolds.

The morphisms to curves appearing in Problem~\ref{pbm:geometricmorphism}
are admissible maps in the following sense (see~\cite{Arapura-geometry}).

\begin{dfn}[Admissible map]\label{dfn:admissible}
	Let $U$ be a Zariski open set inside a compact K\"ahler manifold. Let $C$ be a smooth quasi-projective curve. A map $f:U\to C$ is called \emph{admissible} if it satisfies the following
	conditions:
	\begin{itemize}
		\item It is holomorphic and surjective.
		\item It has a holomorphic enlargement $\widehat f:\widehat X\to \overline C$ with connected fibers, where 
		$\widehat X$ is a K\"ahler compactification of $U$ such that $\widehat X\setminus U$ is a simple
		normal crossings divisor and $\overline C$ is a smooth compactification of $C$.
	\end{itemize} 
\end{dfn}

\begin{rem}[Geometric Realizability Problem in the algebraic setting]
	\label{rem:GMPalgebraic}
	Let $C$ be a smooth quasi-projective curve. If $U$ is a smooth quasi-projective variety, any admissible map
	$f:U\to C$ is algebraic by Theorem~\ref{thm:holomorphicExt} below and by \cite[Thm. 8.5]{Shafarevich2}. Conversely,
	any surjective algebraic morphism $f:U\to C$ with connected generic fibers is admissible. Thus, in the
	quasi-projective setting, Problem~\ref{pbm:geometricmorphism} asks about the algebro-geometric realization of
	epimorphisms of quasi-projective groups to curve groups with finitely generated kernel.
\end{rem}

\begin{rem}
	Let us clarify the use of the phrase ``up to isomorphism in the target'' in Problem~\ref{pbm:geometricmorphism}.
	Note that $G$ is an abstract group which is not canonically identified with the fundamental group of a particular
	curve with respect to a particular base point. Also note that not any identification between $G$ and the
	fundamental group of a curve will yield an algebraically realizable epimorphism. For example, the general curve
	of genus $g\geq 3$ is automorphism free. However, the extended mapping class group of a smooth projective curve
	$C$ of genus $g\geq 3$ (which is isomorphic to $\mathrm{Out}(\pi_1(C))$ by the Dehn-Nielsen-Baer theorem) is an
	infinite group. Thus, there are automorphisms of the fundamental group of a general curve of genus $g\geq 3$
	which cannot be realized algebraically.
\end{rem}

Since the spaces are path connected, a change of base point in $U$ induces isomorphisms that are compatible
with the geometric homomorphism $F_*$, so there is no need to specify the base point in the
Geometric Realizability Problem~\ref{pbm:geometricmorphism} (see Remark~\ref{rem:basepoint}).

Let $C_{g,r}$ denote a smooth quasi-projective curve of genus $g$ with $r$ punctures.
The following result, due to Catanese \cite{Catanese-Fibred}, which builds on contributions
by other authors (see~\cite{Siu-strong,Bauer-irrational}), summarizes the state of the art regarding the
Geometric Realizability Problem~\ref{pbm:geometricmorphism}
for compact K\"ahler manifolds with respect to fundamental groups of compact curves $C_{g,0}$ of general
type ($g\geq 2$) and for proper Zariski open sets in compact K\"ahler manifolds with respect to free groups of
rank~$s\geq 2$.

\begin{thm}[Catanese]\label{thm:summaryCatanese}
Let $U$ be a compact K\"ahler manifold (resp. a proper Zariski open set in a compact K\"ahler manifold) and
let $\psi:\pi_1(U)\to G$ be an epimorphism with finitely generated kernel, where $G$ is the fundamental group
of a smooth projective curve of genus $g\geq 2$ (resp. a free group $\FF_s$ with $s\geq 2$). Then there
exists an admissible map $F:U\to C_{g}$, where $C_{g}$ is a smooth projective curve of genus $g$ (resp.
$F:U\to C_{g,r}$, where $C_{g,r}$ is a smooth quasi-projective curve of genus $g$ with $r=s+1-2g$ points removed)
with no multiple fibers such that $\psi$ coincides with $F_*$ up to isomorphism in the target.
\end{thm}

The proof if $U$ is compact can be found in \cite[Theorem 4.3]{Catanese-Fibred}, and the non-compact case
follows from \cite[Lemma 3.2, Theorem 5.4]{Catanese-Fibred}. Note that Theorem~\ref{thm:summaryCatanese}
does not address the Geometric Realizability Problem~\ref{pbm:geometricmorphism} in the mixed case
where $U$ is non-compact and $G$ is the fundamental group of a smooth projective curve of genus $g\geq 2$,
where an epimorphism $\psi:\pi_1(U)\to G$ with finitely generated kernel could in principle be realized by an
admissible map. Also note that Theorem~\ref{thm:summaryCatanese} does not provide a relationship between the
existence of surjective algebraic morphisms to curves which have multiple fibers and certain quotients of~$\pi_1(U)$.

The goal of this paper is to address the following generalization of the Geometric Realizability
Problem~\ref{pbm:geometricmorphism}, which uses the language of orbifold fundamental groups and orbifold morphisms.
This generalization includes morphisms with multiple fibers and morphisms from non-compact manifolds to compact
curves.

\begin{problem}[Orbifold Geometric Realizability Problem]
\label{pbm:orbifoldgeometricmorphism}
Let $U$ be a (non-empty, connected) Zariski open subset of a compact K\"ahler manifold and let
$\psi:\pi_1(U)\to G$ be an epimorphism with finitely generated kernel, where $G$ is a group which is
isomorphic to the orbifold fundamental group of a smooth quasi-projective curve with an extra orbifold structure.

Determine under what conditions there exists an admissible map
$F:U\to C$ to a smooth quasi-projective curve $C$ realizing $\psi$, that is, such that if $C$ is endowed with
its maximal orbifold structure with respect to $F$, $\psi$ and $F_*:\pi_1(U)\to G=\pi_1^{\orb}(C)$ coincide up
to isomorphism in the target.
\end{problem}

By an orbifold structure on a smooth quasi-projective curve we simply mean a choice of a finite number of points
$M_+=\{P_1,...,P_n\}\subset C$ and an integer label $m_i\geq 2$ at each point $P_i\in M_+$.
The orbifold fundamental group $\pi_1^{\orb}(C)$ is defined based on these choices (see Definition~\ref{dfn:orbipi1}
and Remark~\ref{rem:orbipi1open}).
A morphism from a quasi-projective manifold $U$ to an orbifold curve $C$ is a way to encode the
multiple fibers of $F$, that is, $F$ is required to have a multiple fiber at $P_i$ of multiplicity
(a multiple of) $m_i$ (see Definition~\ref{dfn:orbimorphism}). The finitely generated kernel assumption
on $\psi$ is necessary for the existence of one such $F$, as seen in Remark~\ref{rem:fg}.

\begin{rem}\label{rem:converse}
It is known that, if $F:U\to C$ is an admissible map to a smooth quasi-projective curve, then
there exists a finitely generated normal subgroup $K\trianglelefteq\pi_1(U)$ such that
$\pi_1(U)/K\cong\pi_1^{\orb}(C)$ for some orbifold structure on $C$
(see Remarks~\ref{rem:SteinAdmissible},~\ref{rem:inducedorb} and~\ref{rem:fg} in the preliminaries
Section~\ref{s:preliminaries}). The Orbifold Geometric Realizability Problem~\ref{pbm:orbifoldgeometricmorphism}
addresses not only the converse of this statement, but also whether or not the quotient map
$\pi_1(U)\twoheadrightarrow \pi_1(U)/K$ is realized by an admissible map to a curve.
\end{rem}

If $G$ is a finitely generated free product of cyclic groups (i.e. the orbifold fundamental group of a
non-compact curve), Problem~\ref{pbm:orbifoldgeometricmorphism} was addressed in \cite{ji-Eva-orbifold} by the authors in the case
where $\psi$ is an isomorphism.

\begin{thm}[{\cite[Theorem 1.3, Remark 1.5]{ji-Eva-orbifold}}]\label{thm:ji-Eva}
	Let $U$ be a connected proper Zariski open set in a compact K\"ahler manifold. Let $s,n\geq 0$, $m_1,\ldots,m_n\geq 2$. Suppose that $\pi_1(U)\cong\FF_s*\ZZ_{m_1}*\ldots*\ZZ_{m_n}$, and that $\pi_1(U)$ is an infinite group. Then, there exists a smooth quasi-projective curve $C$ of genus $g$ with $r$ points removed and an admissible map $F:U\to C$ such that
	\begin{itemize}
		\item $2g+r-1=s$,
		\item $F$ has exactly $n$ multiple fibers, with multiplicities $m_1,\ldots,m_n$.
		\item $F_*:\pi_1(U)\to\pi_1^{\orb}(C)$ is an isomorphism, where $C$ is endowed with its maximal
		orbifold structure with respect to $F$, corresponding to the tuple of multiplicities $\bar m=(m_1,\ldots,m_n)$.
	\end{itemize}
\end{thm}

In order to identify a class of groups for which to solve the Orbifold Geometric Realizability Problem~\ref{pbm:orbifoldgeometricmorphism}, we consider the following definition.

\begin{dfn}[Orbifold Euler characteristic]\label{dfn:orbifoldEulerChar}
	Let $C_g$ be a smooth projective curve of genus $g$, and let $C_{g,(r,\bar m)}$
	denote the orbifold structure on the curve $C_g$ with $r$ punctures and $n$ multiple points of
	multiplicities $\bar m=(m_1,\ldots,m_n)$. The orbifold Euler characteristic of $C_{g,(r,\bar m)}$
	is the quantity
	$$\chi_{g,(r,\bar m)}:=2-2g-r-\sum_{i=1}^{n}\left( 1-\frac{1}{m_i}\right).$$
\end{dfn}

Let $g,r,n\geq 0$, and let $\GG_{g,(r,\bar m)}$ be the orbifold fundamental group of $C_{g,(r,\bar m)}$.
We refer to these groups as \emph{curve orbifold groups}. The class of curve orbifold groups
contains interesting families of groups, such as all finitely generated free products of cyclic groups
(for $r\geq 1$, see Remark~\ref{rem:freeProduct}) and all triangle groups (for $g=0$, $r=0$ and $n=3$,
see Remark~\ref{rem:triangle}).

\begin{rem}\label{rem:chiInvariant}
	If $\GG_{g,(r,\bar m)}$ is not a finite cyclic group, the quantity $\chi_{g,(r,\bar m)}$ is an invariant
	of the isomorphism class of the group $\GG_{g,(r,\bar m)}$. This follows from  Remark~\ref{rem:freeProduct} and
	Theorem~\ref{thm:CCOnotFreeProduct}. Moreover, if $\GG_{g,(r,\bar m)}$ is a finite cyclic group, then
	$\chi_{g,(r,\bar m)}>0$ (see List~\ref{list}). Hence, the sign of $\chi_{g,(r,\bar m)}$ is an invariant
	of the isomorphism class of the group $\GG_{g,(r,\bar m)}$.
\end{rem}

The main result of the paper is the following theorem (see Theorem~\ref{thm:GOP} for more details), which unifies and extends Theorems~\ref{thm:summaryCatanese} and \ref{thm:ji-Eva}.
\begin{thm}\label{thm:GOPintro}
The Orbifold Geometric Realizability Problem~\ref{pbm:orbifoldgeometricmorphism} has a positive answer for the curve orbifold groups $\GG_{g,(r,\bar m)}$ such that
$\chi_{g,(r,\bar m)}<0$. Moreover, the multiplicities of the multiple fibers of $F$ coincide with the entries of $\bar m$, and one such $F$ is uniquely determined up to isomorphism of algebraic varieties in the target.
\end{thm}

Note that, unlike Theorem~\ref{thm:GOPintro}, Theorem~\ref{thm:ji-Eva} also includes the cases where the orbifold fundamental group is isomorphic to $\ZZ$ or to $\ZZ_2*\ZZ_2$ (which have $0$ orbifold Euler characteristic), but it needs the extra assumption that $\psi$ is an isomorphism.

\begin{rem}
	All the curve orbifold groups satisfying that $\chi_{g,(r,\bar m)}<0$ are infinite groups (see Remarks~\ref{rem:freeProduct} and~\ref{rem:infinite}).
	 The condition $\chi_{g,(r,\bar m)}<0$ is quite general, see List~\ref{list} in Section~\ref{ss:cogroups} for the list of groups $\GG_{g,(r,\bar m)}$ not satisfying that
	$\chi_{g,(r,\bar m)}<0$ if $r=0$, and Remark~\ref{rem:freeProduct} if $r\geq 1$. 
	
	In particular, $\chi_{g,(r,\bar m)}<0$ for every finitely generated free product of cyclic groups other than $\ZZ$,
	$\ZZ_2*\ZZ_2$, or the finite cyclic groups. Also, $\chi_{g,(r,\bar m)}<0$ for every hyperbolic triangle group (the triangle groups not mentioned in List~\ref{list}).
\end{rem}

\begin{rem}\label{rem:general_type}
	Note that, if $C$ is a smooth quasi-projective curve of genus $g$ with $r$ punctures without an
	orbifold structure ($\bar m=\emptyset$), then the condition $\chi_{g,(r,\emptyset)}=\chi(C)<0$ is known in
	the literature as a curve of \emph{general type}.
\end{rem}

The structure of this paper is as follows. Section~\ref{s:preliminaries} contains some preliminaries in both geometry (K\"ahler manifolds, characteristic varieties, orbifolds and the exact sequence in fundamental groups associated to an admissible map) and group theory (curve orbifold groups and the \NINF\ property) which are necessary for the proof of the main theorem. As we already mentioned, the main result of this paper, Theorem~\ref{thm:GOPintro}, yields that the Orbifold Geometric Realizability Problem~\ref{pbm:orbifoldgeometricmorphism} has a positive answer for curve orbifold groups with negative orbifold Euler characteristic. This is proved in Section~\ref{s:main} (Theorem~\ref{thm:GOP}). 

In Section~\ref{s:sharp} we show that the main theorem is sharp, in the sense that the Orbifold Geometric Realizability Problem~\ref{pbm:orbifoldgeometricmorphism} does not have a positive answer for any curve orbifold group with non-negative orbifold Euler characteristic. More concretely, for each curve orbifold group $G$ with non-negative orbifold Euler characteristic, we find a quasi-projective variety $U$ and an epimorphism $\psi:\pi_1(U)\to G$ with finitely generated kernel which cannot be realized geometrically by an admissible map. 
In Section~\ref{s:examples} we discuss some interesting examples satisfying the hypotheses of the main theorem, as well as their realizations by admissible maps. 

Finally, in Section~\ref{s:planecurves}, Theorems~\ref{thm:realizable-spherical-euclidean} and~\ref{thm:upToSextic}, we apply the main theorem to study 
the following problem for the class of curve orbifold groups.

\begin{problem}[Serre's question in $\PP^2$]\label{prob:serre}
Given a group $G$, can $G$ be realized as the fundamental group of a curve complement in $\PP^2$? That is,
is there a plane curve $C\subset\PP^2$ such that $G\cong\pi_1(\PP^2\setminus C)$?
\end{problem}

Note that all curve orbifold groups are quasi-projective (see Proposition~\ref{prop:realization}).
The answer to Problem~\ref{prob:serre} for the class of curve orbifold groups is given in Corollary~\ref{cor:Serre} except for a special
class of hyperbolic triangle groups, as shown below.

\begin{center}
	\begin{figure}[ht]
		\begin{tabular}{|c|c|c|c|}
			\hline
			\cellcolor{grisclarito}$\GG_{g,(r,\bar m)}$ & \cellcolor{grisclarito}\textsc{yes} & \cellcolor{grisclarito}\textsc{open problem} & \cellcolor{grisclarito}\textsc{no} \\\hline
			\cellcolor{grisclarito}$r\geq 1$ &
			$\array{c}\FF_s*\ZZ_p*\ZZ_q\\
			{\scriptscriptstyle{p,q\geq 1,\ \gcd(p,q)=1,\ s\geq 0}}\endarray$ & \textsc{none}
			& \textsc{otherwise}\\\hline
			\cellcolor{grisclarito}$r=0$ & $\ZZ_p$, ${\scriptscriptstyle{p\geq 1}}$ and $\ZZ^2$ & 
			$\array{c}{\scriptscriptstyle{g=0,\ \bar m=(m_1,m_2,m_3),\ \chi_{g,(r,\bar m)}<0,}}\\
			{\scriptscriptstyle{\gcd(m_1,m_2)=1,\ \gcd(m_1m_2,m_3)\geq 7,}}\\
			{\scriptscriptstyle{\textrm{ and not satisfying the hypotheses of Lemma~\ref{lem:someEvenDegree}}}}
			\endarray$
			& \textsc{otherwise}\\\hline
		\end{tabular}
		\caption{Serre's question for curve complements in $\PP^2$ and curve orbifold groups}
		\label{fig:serre}
	\end{figure}
\end{center}

\subsection*{Acknowledgments}
The authors would like to thank Moisés Herradón Cueto, Anatoly Libgober and Leo Margolis for useful conversations.

\section{Preliminaries}\label{s:preliminaries}
\subsection{Zariski open sets in compact K\"ahler manifolds}

In this section we summarize some well-known properties of Zariski open sets inside compact K\"ahler manifolds and
of morphisms between them which will be used throughout the paper. We include references and/or justifications
for completeness and in order to fix notation.

\begin{rem}[Resolution of singularities]\label{rem:resolution}
Hironaka's resolution of singularities is also available in the complex analytic setting (also in the embedded case),
see \cite[Theorems 2.0.1 and 2.0.2]{Wlodarczyk}. In particular, if $U$ is a Zariski open subset inside of a compact
K\"ahler manifold, then it can also be seen as the complement of a simple normal crossings divisor in another compact
K\"ahler manifold.
\end{rem}

\begin{thm}\label{thm:holomorphicExt}
	Let $U$ be a Zariski open set inside a compact K\"ahler manifold $X$, let $C$ be a smooth quasi-projective
	curve, and let $f:U\to C$ be a surjective holomorphic map. Let $\overline C$ be a smooth (projective)
	compactification of $C$. Then, there exists a compactification $\widehat X$ of $U$ such that $\widehat X$
	is a compact K\"ahler manifold, $\widehat X\setminus U$ is a simple normal crossings divisor, and $f$ extends
	to a holomorphic map $\widehat f:\widehat X\to\overline C$.
\end{thm}

\begin{proof}
By \cite[Corollary A]{Griffiths-ZariskiOpen}, $f$ defines a meromorphic map $f:X\dashrightarrow \overline C$.
By \cite{Remmert} (see also \cite[Theorem 2.5]{Ueno-book}), the set of points of indeterminacy of
$f:X\dashrightarrow \overline C$ has codimension at least 2 on $X$. The proof of the resolution of indeterminacies
in the algebraic setting also works for $f$ by Remark~\ref{rem:resolution}.
\end{proof}

\begin{thm}[Stein factorization for maps to curves]\label{thm:Stein}
	Let $X,C$ be smooth compact complex manifolds, where $C$ is a (necessarily projective) curve, and let  $f:X\to C$ be a surjective holomorphic map. Then, there exists a smooth projective curve $D$ and holomorphic maps $g:X\to D$ and $h:D\to C$ such that $f=h\circ g$, $h$ is a finite map, $g$ is surjective and all the fibers of $g$ are connected. 
\end{thm}
\begin{proof}
	This is a consequence of the more general statement in \cite[p. 219]{GR-sheaves}.
\end{proof}

\begin{rem}\label{rem:SteinAdmissible}
	Let $f:U\to C$ be a surjective holomorphic map between a Zariski open subset $U$ of a compact K\"ahler manifold and a smooth quasi-projective curve $C$. Let $\widehat f:\widehat X\to \overline C$ be an extension of $f$ as in Theorem~\ref{thm:holomorphicExt}. Let $\widehat f=h\circ g$ be its Stein factorization given by Theorem~\ref{thm:Stein}. Then, $g_|:U\to g(U)$ is an admissible map. 
\end{rem}

\begin{thm}[Verdier's generic fibration theorem for maps to curves]\label{thm:Verdier}
	Let $U$ be a Zariski open set in a compact K\"ahler manifold $X$, let $C$ be a smooth quasi-projective curve, and let $f:U\to C$ be a surjective holomorphic map. Then, there exists a finite set $B_f\subset C$, called the set of atypical values of $f$, such that $f_|: f^{-1}(C\setminus B_f)\to C\setminus B_f$ is a locally trivial fibration.
\end{thm}

\begin{proof}
	This is a consequence of \cite[Corollaire 5.1]{Verdier-Sard} applied to a holomorphic enlargement $\widehat f: \widehat X\to \overline C$ as in Theorem~\ref{thm:holomorphicExt}, where $\widehat X\setminus U$ is a simple normal crossings divisor. Indeed, \cite[Corollaire 5.1]{Verdier-Sard} is stated in the algebraic setting, but the proof given therein works in this setting using the Whitney stratification of $\widehat X$ given by $U$ (the open stratum) and the natural stratification of the simple normal crossings divisor $\widehat X\setminus U$.
\end{proof}

\begin{thm}\label{thm:finitecover}
	Let $U$ be a Zariski open subset inside a compact K\"ahler manifold $X$, and let $p:\widetilde U\to U$ be a finite covering map. Then, there exists a compact K\"ahler manifold $Y$ such that $\widetilde U\subset Y$ is a Zariski open subset. 
\end{thm}

\begin{proof}
	By Remark~\ref{rem:resolution}, we may assume that $X\setminus U$ is a simple normal crossings divisor. By the Grauert-Remmert theorem~\cite{GrauertRemmert} (see the statement in \cite[Theorem 1]{NambaGerms}), $p$ extends to a branched covering $\overline p: Z\to X$, where $Z$ is a normal variety containing $\widetilde U$ as a Zariski open subset. Let $\pi:Y\to Z$ be a resolution of singularities of $Z$ as in Remark~\ref{rem:resolution}. Then, by \cite[Proposition 1.3.1 (ii), (v), (vi)]{Varouchas}, $Y$ is a compact K\"ahler manifold containing $\widetilde U$ as a Zariski open subset.
\end{proof}

\subsection{Characteristic varieties}\label{ss:characteristic}
	Let $G$ be a finitely presented group, and let $X$ be a connected finite CW-complex such that $G\cong\pi_1(X)$.
	Let us denote $H:=H_1(X;\ZZ)=G/G'$. The space of characters on $G$ is a complex torus
$
\TT_G:=\Hom(G,\CC^*)=\Hom(H,\CC^*)=H^1(X;\CC^*).
$
This $\TT_G$ can have multiple connected components, but it only contains one connected torus, which we
denote by $\TT_G^{\one}$.

\begin{dfn}
	\label{def-char-var}
The $k$-th~\emph{characteristic variety}~of $G$ is defined by:
	\[
	\Sigma_{k}(G):=\{ \xi \in \TT_G\mid \dim H^1(X,\CC_{\xi}) \ge k \},
	\]
	where $H^1(X,\CC_{\xi})$ is classically called the 
	\emph{twisted cohomology of $X$ with coefficients in the local system $\xi\in \TT_G$}.
	It is also customary to use $\Sigma_{k}(Y)$ for $\Sigma_{k}(G)$ whenever $\pi_1(Y)\cong G$.
\end{dfn}

In order to prove Theorem~\ref{thm:GOPintro}, we will use the following result of Arapura relating the
(first) characteristic variety of $U$ with the existence of admissible maps from $U$ to curves of general type
(see Remark~\ref{rem:general_type}).

\begin{prop}[{\cite[Prop V.1.7]{Arapura-geometry}}]\label{prop:Arapura}
Let $U$ be a Zariski open set of a compact K\"ahler manifold.
\begin{itemize}
	\item If $F:U\to C$ is an admissible map to a smooth quasi-projective curve $C$ of general type,
	then $F^*(H^1(C,\CC^*))$ is a component of $\Sigma_1(U)$. Moreover, if there exists another admissible map
	$F':U\to C'$  to a smooth quasi-projective curve $C'$ of general type such that $F^*(H^1(C,\CC^*))=(F')^*(H^1(C',\CC^*))$, then there exists a (necessarily unique) algebraic isomorphism $\phi:C\to C'$ for which $F'=\phi\circ F$.
	\item Conversely, any positive dimensional
	component of $\Sigma_1(U)$ containing $1$ is of the form $F^*(H^1(C,\CC^*))$ for some admissible map $F:U\to C$ to a smooth quasi-projective curve $C$ of general type.
\end{itemize}
\end{prop}

\subsection{Orbifolds and orbifold fundamental groups.}\label{ss:orbifold}
Consider a smooth projective curve $C_g$ of genus $g$ and choose a labeling map
$\varphi:C_g\to\ZZ_{\geq 0}$ such that $\varphi(P)\neq 1$ only for a finite number of points.
In this context, we will refer to $\varphi$ as an \emph{orbifold structure on} $C_g$.

Given the orbifold structure $\varphi$ on $C_g$, let $M:=\{P\in C_g\mid \varphi(P)\neq 1\}=M_0\cup M_+$, where
$M_0:=\{P\in C_g\mid \varphi(P)= 0\}$ and $M_+=\{P\in C_g\mid \varphi(P)\geq 2\}$. This orbifold structure will
also be denoted by $C_{g,(r,\bar m)}$, where $r=\# M_0$, and, if $M_+=\{P_1,\ldots,P_n\}$,
$\bar{m}=(\varphi(P_1),\ldots,\varphi(P_n))$.

\begin{dfn}[Orbifold fundamental group]\label{dfn:orbipi1}
The \emph{orbifold fundamental group} associated with $C_{g,(r,\bar m)}$,
denoted by $\pi_1^{\orb}(C_{g,(r,\bar m)})$, is the quotient of
$
\pi_1(C_g\setminus M)$ by the normal closure of the subgroup
$
\langle \mu_P^{\varphi(P)}, P\in M_+\rangle,
$
where $\mu_P$ is a meridian in $C_g\setminus M$ around $P\in M$.
\end{dfn}

\begin{rem}\label{rem:orbipi1open}
Let $C'$ be a smooth quasi-projective curve, and let $\varphi: C'\to\ZZ_{\geq 1}$ be a labeling map such that
$\varphi (P)\neq 1$ for only a finite number of points. Let $\bar m$ be the tuple whose entries are the
$\varphi(P)$ that are greater or equal than $2$ for $P\in C'$. Let $C$ be a smooth compactification of $C'$.
Let $g$ be the genus of $C$, and let $r=\# C\setminus C'$, so, in the notation used before, $C=C_g$. Let $\widehat\varphi: C\to\ZZ_{\geq 0}$
be the extension of $\varphi$ to $ C$ such that $\varphi(P)=0$ for all $P\in C\setminus C'$, and let
$C_{g,(r,\bar m)}$ be the corresponding orbifold structure on $C=C_g$. In such a case, where the genus of $C$ is known, the notation $\pi_1^{\orb}(C'_{\bar m})$ or $\pi_1^{\orb}(C_{r,\bar m})$ will refer to $\pi_1^{\orb}(C_{g,(r,\bar m)})$.
\end{rem}

Note that $\pi_1^{\orb}(C_{g,(r,\bar m)})$ has a presentation with generators
$
\{a_i,b_i\}_{i=1,\dots,g} \cup \{\mu_P\}_{P\in M}
$
and relations
\begin{equation}
	\label{eq:rels}
	\mu_P^{m_P}=1, \quad \textrm{ for } P\in M_+, \quad \textrm{ and } \quad
	\prod_{P\in M}\mu_P=\prod_{i=1,\dots,g}[a_i,b_i]
\end{equation}
for appropriately chosen
$\{a_i,b_i\}_{i=1,\dots,g}$ and meridians $\{\mu_P\}_{P\in M}$.

\begin{dfn}[Orbifold morphism]\label{dfn:orbimorphism}
Let $U$ be a Zariski open subset of a compact K\"ahler manifold, and let $C_g$ be a smooth projective curve
of genus $g$. Let $F:U\to C_g$ be a dominant holomorphic map.
Consider the orbifold structure $C_{g,(r,\bar m)}$ given by a certain $\varphi:C_g\to\ZZ_{\geq 0}$ on
$C_g$. We say that $F$ defines an \emph{orbifold morphism} $F:U\to C_{g,(r,\bar m)}$ if, for all
$P\in C_g$ the divisor $F^*(P)$ is a $\varphi(P)$-multiple.

The orbifold $C_{g,(r,\bar m)}$ is said to be \emph{maximal} with respect to $F$ if $M_0=C_g\setminus F(U)$
and for all $P\in F(U)$ the divisor $F^*(P)$ is not an $n$-multiple for any $n>\varphi(P)$.
\end{dfn}

Note that the maximal orbifold structure associated to a surjective holomorphic map as in
Definition~\ref{dfn:orbimorphism} always exists by Theorem~\ref{thm:Verdier}.

The following result is well known (see for instance~\cite[Prop. 1.4]{ACM-multiple-fibers}).
\begin{rem}\label{rem:inducedorb}
Let $F:U\to C_{g,(r,\bar m)}$ be an orbifold morphism. Then, $F$ induces a morphism
$F_*:\pi_1(U)\to\pi_1^{\orb}(C_{g,(r,\bar m)}),$
such that $F_*:\pi_1(U)\to\pi_1(C_g\setminus M_0)$ factors through it and through the natural projection
$\pi_1^{\orb}(C_{g,(r,\bar m)})\twoheadrightarrow \pi_1(C_g\setminus M_0)$. Moreover, if $F$ is admissible,
then $F_*:\pi_1(U)\to\pi_1^{\orb}(C_{g,(r,\bar m)})$ is surjective.
\end{rem}

\subsection{Curve orbifold groups}\label{ss:cogroups}

\begin{dfn}[Curve orbifold group]
A curve orbifold group is a group which admits a presentation of the form
\begin{equation}\label{eq:Fgroup}
    \GG_{g,(r,\bar m)}:=
	\left\langle \array{cc} a_i,b_i,x_j,y_k, 
	{\tiny{\array{c}i=1,...,g\\j=1,..,n\\k=1,...,r\endarray}}\endarray \left|\quad
	\prod_{i=1}^{g}[a_i,b_i]=\prod_{j=1}^{n} x_j\prod_{k=1}^r y_k,\, x_1^{m_1}=\ldots=x_n^{m_n}=1\right.\right\rangle
\end{equation}
for some $g,n,r\geq 0$ and $m_1,\ldots,m_n\geq 2$.
\end{dfn}

Since permutations in $\bar m=(m_1,\ldots,m_n)$ yield isomorphic groups, we will assume $m_1\leq ...\leq m_n$
unless otherwise stated. By Definition~\ref{dfn:orbipi1}, curve orbifold groups are exactly the groups that
appear as orbifold fundamental groups of smooth quasi-projective curves, hence the name. More concretely, in the
notation of Remark~\ref{rem:orbipi1open}, $\GG_{g,(r,\bar m)}\cong\pi_1^{\orb}\left((C_{g,r})_{\bar m}\right)$, where $C_{g,r}$
is a smooth quasi-projective genus $g$ curve with $r$ punctures, and the orbifold structure is given by
$n$ marked points of multiplicities $m_1,\ldots,m_n\geq 2$. Hence, we give the following definition.

\begin{dfn}[Compact and open curve orbifold groups]
If $r=0$, we say that the group $\GG_{g,(r,\bar m)}$ is a \emph{compact curve orbifold group} (\cco ~group for short),
and denote it by $\GG_{g,\bar m}$. Moreover, if $\bar m$ is empty ($n=0$), we will sometimes denote $\GG_{g,-}$ simply
by $\GG_g$. That is, $\GG_g$ is the fundamental group of a smooth projective curve of genus $g$.

If $r\geq 1$, we say that the group $\GG_{g,(r,\bar m)}$ is an \emph{open curve orbifold  group}.
\end{dfn}

\begin{rem}
By Theorem~\ref{thm:CCOnotFreeProduct} below, a group cannot be both isomorphic to a \cco ~group and to an open
curve orbifold group unless it is a finite cyclic group.
\end{rem}

\begin{rem}\label{rem:freeProduct}
Suppose that $r\geq 1$. Then, $\GG_{g,(r,\bar m)}\cong \FF_{2g+r-1}*\ZZ_{m_1}*\ldots*\ZZ_{m_n}$. Hence, the class
of open curve orbifold groups coincides with the class of finitely generated free products of cyclic groups. Moreover,
$$
\chi_{g,(r,\bar m)}>0\quad\Leftrightarrow\quad g=0, r=1, n=0,1 \quad\Leftrightarrow\quad \GG_{g,(r,\bar m)}
\text{ is a finite cyclic group, }
$$
and also
\begin{align*}
\chi_{g,(r,\bar m)}=0\quad&\Leftrightarrow\quad g=0, r=1, n=2, m_1=m_2=2\text{ or }g=0, r=2, n=0 \\
&\Leftrightarrow\quad \GG_{g,(r,\bar m)}\cong \ZZ_2*\ZZ_2 \text{ or }\ZZ.
\end{align*}

In particular, $\chi_{g,(r,\bar m)}\leq 0$ if and only if $\GG_{g,(r,\bar m)}$ is infinite.
\end{rem}

\begin{rem}\label{rem:triangle}
\cco ~groups have the following presentation:
\begin{equation}\label{eq:Ggm}
		\GG_{g,\bar m}:=\left\langle \array{cc} a_i,b_i,x_j,
		{\tiny{\array{c} i=1,...,g\\j=1,...,n\endarray}}\endarray
		\left|\quad \prod_{i=1}^g[a_i,b_i]=\prod_{j=1}^n x_j,\, x_1^{m_1}=\ldots=x_n^{m_n}=1\right.\right\rangle
	\end{equation}
In particular, the \cco\ groups with $g=0$ and $n=3$ are the triangle groups.
\end{rem}

The following result establishes that curve orbifold groups are quasi-projective.
\begin{prop}[{\cite[Prop. 1.19]{ACM-orbifoldgroups}}]\label{prop:realization}
Any curve orbifold group can be realized as the fundamental group of a smooth quasi-projective surface.
\end{prop}

We now shift our focus to \cco\ groups, and use the orbifold Euler characteristic to classify them into three different kinds.
 
\begin{dfn}\label{dfn:3cases}
The \cco\ group $\GG_{g,\bar m}$ is spherical (resp. Euclidean, resp. hyperbolic)
when its orbifold Euler characteristic $\chi_{g,\bar m}:=\chi_{g,(0,\bar m)}=2-2g-\sum_{j=1}^n\left(1-\frac{1}{m_j}\right)$
is $>0$ (resp. $=0$, resp. $<0$).
\end{dfn}

It is not clear from Definition~\ref{dfn:3cases} whether a group could be isomorphic to two \cco\
groups that were not of the same kind. The following remark implies that a spherical \cco\
group cannot be isomorphic to a Euclidean/hyperbolic \cco\ group, and the remaining cases
will be settled in Remark~\ref{rem:notiso}.
\begin{rem}[\cite{HKS-infinite}]\label{rem:infinite}
	$\GG_{g,\bar m}$ is an infinite group if and only if
	$\chi_{g,\bar m}\leq 0$.
\end{rem}

\begin{lst}[Compact curve orbifold groups which are not hyperbolic]\label{list} \ 
	\begin{enumerate}
		\item Spherical groups: They are finite by Remark~\ref{rem:infinite}.
		\begin{itemize}
			\item $\GG_{0,(m_1,\ldots,m_n)}$ for $n=0,1,2$, that is, finite cyclic groups.
			Note that if $g=0$ and $n=0,1$, then  $\GG_{g,\bar m}$ is the trivial group.
			If $g=0$ and $n=2$, then $\GG_{g,\bar m}$ is cyclic of order
			$\gcd(m_1,m_2)$.
			\item $\GG_{0,(m_1,m_2,m_3)}$ such that $\sum_{j=1}^3 \frac{1}{m_j}>1$. These are the
			spherical triangle groups, which are finite, and correspond to the following triples $\bar m$:
			\begin{itemize}
				\item $(2,3,3)$, which yields the alternating group $A_4$,
				\item $(2,3,4)$, which yields the symmetric group $S_4$,
				\item $(2,3,5)$, which yields the symmetric group $A_5$,
				\item $(2,2,2)$, which yields the Klein group $\ZZ_2\times\ZZ_2$, and
				\item $(2,2,n)$ for all $n\geq 3$, which yields the dihedral group of order $2n$.
			\end{itemize}
		\end{itemize}
		\item Euclidean groups: They are infinite by Remark~\ref{rem:infinite}, and all of them are wallpaper groups: $\GG_{0,(2,3,6)}$, $\GG_{0, (2,4,4)}$,
		$\GG_{0, (3,3,3)}$, $\GG_{0, (2,2,2,2)}$, and $\GG_1\cong \ZZ^2$.
	\end{enumerate}
\end{lst}

We recall some key facts about \cco\ groups which are well known to experts.

\begin{rem}\label{rem:propertieshyperbolic}
	Let $\GG_{g,\bar m}$ be a hyperbolic \cco\ group, i.e. a group with
	presentation~\eqref{eq:Ggm} such that $\chi_{g,\bar m}<0$. Then, the following
	properties hold:
	\begin{enumerate}
		\item\label{part:abelian}
		Any abelian subgroup of $\GG_{g,\bar m}$ is cyclic (\cite[Theorem 1]{Greenberg}).
		\item\label{part:finite}
		Any finite subgroup of $\GG_{g,\bar m}$ is cyclic, and conjugate to a subgroup of one
		of the groups $\langle x_j\rangle$ for some $j=1,\ldots,n$ in~\eqref{eq:Ggm}. In particular, any
		element of finite order of $\GG_{g,\bar m}$ is conjugate to a power of one of the $x_j$'s  (\cite[Theorem 1]{Greenberg}).
		\item\label{part:order}
		The order of $x_j$ is $m_j$ (\cite{HKS-infinite}).
		\item\label{part:finitecover}
		$\GG_{g,\bar m}$ is not abelian. Indeed, by Lemma~\ref{lem:Fenchel} below,
		$\GG_{g,\bar m}$ has a normal subgroup of finite index which is isomorphic to the
		fundamental group of a genus $g'>1$ smooth projective curve (see \cite{HKS-infinite} for a
		combinatorial proof), which is not abelian.
		\item\label{part:abelianFI}
		$\GG_{g,\bar m}$ does not have an abelian subgroup of index 2. If it did, it would be
		normal, and by part \eqref{part:abelian}, also cyclic. Its intersection with the normal subgroup of
		part \eqref{part:finitecover} would be a cyclic subgroup of finite index of the fundamental group
		of a genus $g'>1$ smooth projective curve, and this is impossible.
		\item\label{part:finiteNormal}
		$\GG_{g,\bar m}$ does not have any finite nontrivial normal subgroups. Indeed, parts \eqref{part:finitecover} and \eqref{part:abelianFI} imply that $\GG_{g,\bar m}$ is not quasi-abelian (in the notation of \cite{Greenberg}), and the result follows from \cite[Theorem 4]{Greenberg}
		and the fact that $\GG_{g,\bar m}$ is infinite.
	\end{enumerate}
\end{rem}

Some of these properties are also shared in the Euclidean case. 

\begin{rem}\label{rem:propertiesEuclidean}
		Let $\GG_{g,\bar m}$ be a Euclidean \cco\ group, i.e. a group with
		presentation~\eqref{eq:Ggm} such that $\chi_{g,\bar m}=0$.
		Then, the following properties hold:
	\begin{enumerate}
		\item\label{part:orderE}
		The order of $x_j$ is $m_j$: If $n=3$ and $g=0$ this follows by \cite[Theorem 1.2]{ChinyereHowie},
		which says that if $G_j$ is the subgroup of $\GG_{g,\bar m}$ generated by $x_j$, then
		the natural homomorphism $G_j\to \GG_{g,\bar m}$ is injective. The proof for
		$\GG_{0,(2,2,2,2)}$ can be obtained from the previous statement, noting that
		$\GG_{0,(2,2,2)}$ is a quotient. The statement is trivial for $\GG_1$.
		\item\label{part:finiteNormalE}
		$\GG_{g,\bar m}$ does not have any finite nontrivial normal subgroups. Indeed,
		$\GG_{g,\bar m}$ is a wallpaper group, so it can be seen as a subgroup of the group of
		isometries of $\RR^2$, which, in this case, is orientation preserving. If
		$\GG_{g,\bar m}=\GG_1\cong\ZZ^2$, the statement is trivial. Suppose that
		$\GG_{g,\bar m}$ is any of the other four Euclidean \cco\ groups.
		Every element of $\GG_{g,\bar m}$ can be interpreted as a translation (which has infinite
		order) or a rotation, and there exist nontrivial rotations. The (normal) subgroup of translations
		of $G$ is isomorphic to $\ZZ^2$, and it has finite index.  Conjugating a rotation by all the powers of a translation one
		obtains infinitely many different rotations.
	\end{enumerate}
\end{rem}

\begin{rem}\label{rem:notiso}
		As mentioned in Remark~\ref{rem:propertiesEuclidean}, all Euclidean \cco\ groups
		have a subgroup that is isomorphic to $\ZZ^2$, which is not cyclic. By
		Remark~\ref{rem:propertieshyperbolic} part~\eqref{part:abelian}, no Euclidean compact curve
		orbifold group can be isomorphic to a hyperbolic \cco\ group.
\end{rem}

The rest of this section contains some statements about curve orbifold groups which are well-known to experts,
although we include the proofs for clarity. The following lemma explains the geometric significance of the
orbifold Euler characteristic of curve orbifold groups.
\begin{lemma}\label{lem:Fenchel}
	Let $\GG_{g,(r,\bar m)}$ be a curve orbifold group. Let $C_{g,(r,\bar m)}$ be a smooth projective curve of
	genus $g$ endowed with the orbifold structure corresponding to $r$ points of multiplicity $0$ and $n$ marked
	points with multiplicities $\geq 2$ given by the tuple $\bar m=(m_1,\ldots,m_n)$. Let
	$P_j\in C_g$ be the point associated with the multiplicity $m_j$ for all $j=1,\ldots,n$, and let $M_0$ be the
	set of points in $C_g$ with multiplicity $0$. Then, the following hold:
	\begin{enumerate}
	\item There exists a finite index normal subgroup $N$ of $\GG_{g,(r,\bar m)}$ such that $N$ doesn't contain any torsion elements other than the identity.
	\item $N$ induces a ramified cover $C'\to C_g\setminus M_0$ ramified along $\{P_1,\ldots,P_n\}$, where $C'$ is
	a smooth quasi-projective curve such that $\pi_1(C')\cong N$. Moreover, $C'$ is projective if and only if $r=0$,
	and quasi-projective and non-compact otherwise.
	\item The branching divisor of this cover is $\sum_{j=1}^n m_j P_j$, unless perhaps if $r=0$ and
	$\GG_{g,\bar m}=\GG_{g,(r,\bar m)}$ is a spherical \cco\ group.
	\end{enumerate}
	
	Let $\rho$ be defined as follows:
	$$
	\rho:=\left\{\begin{array}{lr}
	\text{genus of }C' & \text{ if $r=0$},\\
	\text{rank of $\pi_1(C')$, which is a free group} & \text{ if $r\geq 1$.}
	\end{array}\right.
	$$
	
	Then,
	\begin{itemize}
		\item $\rho\geq 2$ $\Leftrightarrow$  $\chi_{g,(r,\bar m)}<0$.
		\item $\rho= 1$ $\Leftrightarrow$  $\chi_{g,(r,\bar m)}=0$.
		\item $\rho= 0$ $\Leftrightarrow$  $\chi_{g,(r,\bar m)}>0$.
	\end{itemize}
\end{lemma}

\begin{proof}
The existence of $N$ is known as Fenchel's conjecture, proved in \cite{Nielsen-commutator} for $r>0$,
in \cite{BundgaardNielsen} for $g>0$ and in \cite{Fox-Fenchel} for $g=0$ and $r=0$.

We use facts about branched covers that can be found in \cite[Section 1]{Uludag}. The
normal subgroup $N$ yields a finite Galois cover $C'\to C_g\setminus M_0$ branched along $M_+=\{P_1,\ldots,P_n\}$,
where $\pi_1(C')=N$. In fact, the existence of such $N$ is equivalent to the
existence of such a cover. Note that $C'$ is a normal curve, so it is smooth. $C'$ is compact if and only if
$C_g\setminus M_0$ is compact, which happens if and only if~$r=0$.

The branching divisor is of the form $\sum_{j=1}^n \tilde m_j P_j$, with
$1\leq \tilde m_j\leq m_j$, and $m_j=\tilde m_j$ if and only if $x_j^{m}\neq 1\in\GG_{g,\bar m}$
for all $1\leq m<m_j$. Hence, by Remark~\ref{rem:propertieshyperbolic} part~\eqref{part:order} and
Remark~\ref{rem:propertiesEuclidean} part~\eqref{part:orderE}, the branching divisor is of the form
$\sum_{j=1}^n  m_j P_j$ if $r=0$ and $\GG_{g,(0,\bar m)}$ is hyperbolic or Euclidean. If $r\geq 1$,
$\GG_{g,(r,\bar m)}$ is a free product of cyclic groups and we also obtain that $\tilde m_j=m_j$.

Let $d$ be the degree of the cover $C'\to C_g\setminus M_0$, which is the index of $N\trianglelefteq\GG_{g,(r,\bar m)}$.
The rest of the statements follow from the following two equations. If $r=0$,
$$
2-2\rho=\chi(C')=d(\chi(C_g\setminus M_+))+\sum_{j=1}^n\frac{d}{\tilde m_j}\geq d\, \chi_{g,\bar m},
$$
which, if $\GG_{g,\bar m}$ is hyperbolic or Euclidean, is an equality. If $r\geq 1$,
$$
1-\rho=\chi(C')=d(\chi(C_g\setminus\left(M_0\cup M_+\right)))+\sum_{j=1}^n\frac{d}{m_j}=
d\left(2-2g-r-\sum_{j=1}^n\left(1-\frac{1}{m_j}\right)\right)= d\, \chi_{g,(r,\bar m)}.
$$
\end{proof}
	
\begin{rem}\label{rem:chiIsoClass}
It follows from the proof of Lemma~\ref{lem:Fenchel} that if the \cco\ group $\GG_{g,\bar m}$ is infinite
(that is, hyperbolic or Euclidean), the quantity $\chi_{g,\bar m}$ is an invariant of the isomorphism class
of the group $\GG_{g,\bar m}$. In particular, this gives a different proof of Remark~\ref{rem:notiso}.
\end{rem}

The following result clarifies when the tuple $(g,r,\bar m)$ is an invariant of the isomorphism class of
$\GG_{g,(r,\bar m)}$. 

\begin{thm}\label{thm:CCOnotFreeProduct}
Let $g,g',r,r',n,n'\geq 0$, let $\bar m=(m_1,\ldots,m_n)$ and $\bar m'=(m_1',\ldots,m_{n'}')$ be such that $2\leq m_1\leq m_2\leq\ldots\leq m_n$ and $2\leq m_1'\leq m_2'\leq\ldots\leq m_{n'}'$. Suppose that $\GG_{g,(r,\bar m)}\cong\GG_{g',(r',\bar m')}$ and that those groups are not finite cyclic groups. Then,
\begin{enumerate}
\item If $r=0$ (i.e. $\GG_{g,(r,\bar m)}$ is a \cco\ group), then
	$$
	(g,r,n,\bar m)=(g',r',n',\bar m').
	$$
\item If $r\geq 1$, then $r'\geq 1$, and
$$
(2g+r,n,\bar m)=(2g'+r',n',\bar m').
$$
\end{enumerate}
\end{thm}
\begin{proof}
	Let us first show that, if $r=0$, then $r'=0$. We argue by contradiction. Suppose that $\GG_{g,(r,\bar m)}\cong\GG_{g',(r',\bar m')}$ is not finite cyclic and that $r'>0$. Then, $\GG_{g',(r',\bar m')}\cong\FF_{2g'+r'-1}*\ZZ_{m_1'}*\ldots*\ZZ_{m_{n'}'}$. Let $N$ be a finite index subgroup of $\GG_{g,(r,\bar m)}$ that doesn't contain any torsion elements other than the identity, as given by Lemma~\ref{lem:Fenchel}. Let $N'$ be the normal subgroup of $\GG_{g',(r',\bar m')}$ corresponding to $N$ by the isomorphism $\GG_{g,(r,\bar m)}\cong\GG_{g',(r',\bar m')}$. By Lemma~\ref{lem:Fenchel}, $N$ is the fundamental group of a smooth projective curve, and $N'$ is a free group. Since $N\cong N'$, this is only possible if $N$ and $N'$ are both trivial. In particular, $\GG_{g,(r,\bar m)}\cong\GG_{g',(r',\bar m')}$ is finite, and, since $\GG_{g',(r',\bar m')}$ is a free product of cyclic groups, that implies that $\GG_{g,(r,\bar m)}\cong\GG_{g',(r',\bar m')}$ is a finite cyclic group. Hence, if $r=0$, then $r'=0$. This also shows that if $r\geq 1$, then $r'\geq 1$.
	
	The result for $r,r'\geq 1$ follows from the classification of finitely generated free products of cyclic groups. Hence, it suffices to show the result for $r=r'=0$.
	
	Suppose that $\GG_{g,\bar m}\cong \GG_{g',\bar m'}$. If $\GG_{g,\bar m}$ is finite (but not cyclic), the result follows from List~\ref{list}. From now on,
we will assume that $\GG_{g,\bar m}$ is infinite.

If $\chi_{g,\bar m}=0$, this follows from the classification of wallpaper groups. Hence, we may assume that
$\chi_{g,\bar m}<0$.

The equality $g=g'$ is obtained by looking at the torsion-free parts of the abelianizations of both groups. The equality $n=n'$ can be obtained from \cite[Proposition 2.11]{ACM-characteristic} if $n$ or $n'$ are not in $\{0,1,2\}$. If $n,n'\in\{0,1,2\}$, the equality $n=n'$ can be obtained from Remark~\ref{rem:propertieshyperbolic} (parts \eqref{part:finite} and~\eqref{part:order}) and Remark~\ref{rem:chiIsoClass}. In particular, this
implies the result for $n=0$, and by Remark~\ref{rem:chiIsoClass}, also for $n\geq 1$. Suppose that the result is true for a certain $n\geq 1$, and that $\GG_{g,(m_1,\ldots,m_{n+1})}\cong \GG_{g,(m_1',\ldots,m_{n+1}')}$, where $\chi_{g,\bar m}<0$. By parts~\eqref{part:finite} and~\eqref{part:order} in
Remark~\ref{rem:propertieshyperbolic} $m_n=m'_n$, since that is the maximum order that a torsion element can
have in the group. The equality $m_n=m'_n$ and Remark~\ref{rem:propertieshyperbolic}\eqref{part:finite} can be used to show that
$\GG_{g,(m_1,\ldots,m_{n})}\cong \GG_{g,(m_1',\ldots,m_{n}')}$. If $n\geq 3$ or $g>0$, $\GG_{g,(m_1,\ldots,m_{n})}$ is not a cyclic groups and the result follows by induction.
Hence, it suffices to prove that, if $\GG_{0,(m_1,m_2,m_3)}\cong \GG_{0,(m_1',m_2',m_3')}$ are
hyperbolic for some $m_1$, $m_2$, $m_3$, $m_1'$, $m_2'$, $m_3'\in\ZZ_{\geq 2}$ with $m_1\leq m_2\leq m_3$ and
$m_1'\leq m_2'\leq m_3'$, then $m_i=m_i'$ for all $i=1, 2,3$. We already know that $m_3=m_3'$. By
Remark~\ref{rem:chiIsoClass}
\begin{equation}\label{eq:isoCCO1}
\frac{1}{m_1}+\frac{1}{m_2}=\frac{1}{m_1'}+\frac{1}{m_2'}.
\end{equation}
Also, since $\GG_{0,(m_1,m_{2})}\cong \GG_{0,(m_1',m_{2}')}$, List~\ref{list} implies that
$
\gcd(m_1,m_2)=\gcd(m_1',m_2')
$. 
Let $d$ be this greatest common divisor, and let $n_i=\frac{m_i}{d}$ and $n_i'=\frac{m_i'}{d}$ for $i=1,2$. Equation~\eqref{eq:isoCCO1} yields that $$\frac{n_1+n_2}{n_1n_2}=\frac{n_1'+n_2'}{n_1'n_2'},$$ and since both fractions are reduced,
\begin{equation}\label{eq:isoCCO2}
n_1+n_2=n_1'+n_2',\quad n_1n_2=n_1'n_2'.
\end{equation}
Since $\gcd(n_1,n_2)=\gcd(n_1',n_2')=1$, the last equation in \eqref{eq:isoCCO2} implies that there exist pairwise coprime positive integers $p,q,r,s$ such that $n_1=pq$, $n_2=rs$, $n_1'=pr$, $n_2'=qs$, in which case the first equation in \eqref{eq:isoCCO2} implies that $p(q-r)=s(q-r)$, which can only happen if $q=r=1$ or if $p=s=1$. Hence, $n_i=n_i'$ for $i=1,2$, which implies $m_i=m_i'$ for $i=1,2$.
\end{proof}

\begin{lemma}\label{lem:hopfian}
	Curve orbifold groups are Hopfian.
\end{lemma}
\begin{proof}
	Let $G$ be a curve orbifold group. By Lemma~\ref{lem:Fenchel}, $G$ contains a finite index normal subgroup $N$ which is either a free group or the fundamental group of a smooth projective curve of genus $g\geq 0$. In particular, $N$ is residually finite, so $G$ is also residually finite and, since $G$ is finitely generated, it is Hopfian (cf.~\cite[Thm. 4.10]{Stillwell-combinatorial}).
\end{proof}

\subsection{Property \NINF}
The notion of \NINF\ groups is due to Catanese and it will be very useful
in the rest of the paper. In order to set definitions and basic consequences, we present a brief summary
of this property following~\cite{Catanese-Fibred}.
\begin{dfn}
A group $G$ satisfies property \NINF\ if every normal non-trivial subgroup of infinite index is not finitely generated.
\end{dfn}

\begin{exam}\label{exam:NINF}
	The free group $\FF_s$ satisfies property \NINF\ for any $s\geq 0$ (\cite[Lemma 3.3]{Catanese-Fibred}).
	Also, $\GG_g$ satisfies property \NINF\ for any $g\geq 2$ (\cite[Lemma 3.4]{Catanese-Fibred}).
\end{exam}

The following result exemplifies the usefulness of the \NINF\ property.

\begin{lemma}[Lemma 3.2 in \cite{Catanese-Fibred}]\label{lem:NINF}
	Let $G$ be an infinite group satisfying property \NINF, and let $1\to A \to B\xrightarrow{\varphi} C\to 1$
	be a short exact sequence of groups such that $A$ is finitely generated, $C$ is infinite and $\varphi$ factors
	as $\varphi=\rho\circ\psi$, where $\psi:B\to G$ is surjective. Then, $\rho:G\to C$ is an isomorphism.
\end{lemma}

\begin{lemma}\label{lem:FgroupNINF}
	Let $G$ be a curve orbifold group which is not a Euclidean \cco\ group. Then, $G$ satisfies property \NINF.
\end{lemma}
\begin{proof}
	If $G$ is a	spherical \cco\ group, the result is trivial because $G$ is finite.
	Hence, from now on, we may assume that $G$ is an infinite curve orbifold group which is not a Euclidean \cco\ group. Note that, in those cases, $G$ does not contain any finite normal subgroups: if $G$ is a hyperbolic \cco\ group, this follows from Remark~\ref{rem:propertieshyperbolic}, part~\eqref{part:finiteNormal}, and if $G$ is a finite free product of cyclic groups, the result follows from the Kurosh subgroup theorem.
	
	We start by showing that, if $G$ is an infinite curve orbifold group, then it
	has a normal subgroup $N$ of finite index satisfying property \NINF\ such that every non-trivial
	normal subgroup $H$ satisfies that $N\cap H$ is non-trivial. We distinguish two cases. By Lemma~\ref{lem:Fenchel}, there exists a finite index normal subgroup $N$ of $G$ which contains no torsion elements. Moreover, if $G$ is an open curve orbifold group (i.e. a finitely generated free product of cyclic groups), then $N$ is a free group of finite rank, and if $G$ is a hyperbolic \cco\ group, then $N\cong \GG_g$ for some $g\geq 2$. In both of these cases, $N$ satisfies property \NINF. Now, let $H$ be a non-trivial normal subgroup of $G$. If $N\cap H$ was trivial, then there would exist a monomorphism from $H$ to the finite group $G/N$, and we obtain a contradiction.
	
	Let $H$ be a non-trivial normal subgroup of $G$ of infinite index, and let $N$ be as in the previous paragraph. Since $N\cap H$ is a normal subgroup of $N$ of infinite index, the subgroup $N\cap H$
	must be infinitely generated. Consider the short exact sequence
	$$
	1\to N\cap H\to H \to H/N\cap H\to 1,
	$$
	where $H/N\cap H$ is finite. If $H$ was finitely generated, so would $N\cap H$, so $H$ is infinitely generated.
\end{proof}

\subsection{The exact sequence associated to an admissible map}

This result is well known in different settings (cf. \cite[Lemma 4.2]{Catanese-Fibred},
\cite[Corollary 2.22]{ji-Eva-orbifold}), but we include it here for the sake of completeness and to unify notation.
\begin{lemma}\label{lem:exact}
	Let $U$ be a Zariski open subset of a compact K\"ahler manifold, let $C_g$ be a smooth projective curve of genus
	$g$ and let $C_{g,r}$ an $r$-punctured $C_g$. Consider $F:U\to C_{g,r}$ an admissible map.
	Let $C_{g,(r,\bar m)}$ be the maximal orbifold structure of $C_{g,r}$ with respect to $F$ for some
	$\bar m=(m_1,\dots,m_n)$. Then, the following sequence is exact
	$$
	\pi_1(F^{-1}(P))\to\pi_1(U)\xrightarrow{F_*}\pi_1^{\orb}(C_{g,(r,\bar m)})\to 1,
	$$
	where the first arrow is induced by the inclusion of a generic fiber $F^{-1}(P)$ over~$P\in C_{g,r}$.
\end{lemma}

\begin{proof}
The case where $r\geq 1$ is proved in \cite[Corollary 2.22]{ji-Eva-orbifold}. 

Suppose that $r=0$. Let $B_F$ be the (finite) set of atypical values of $F$, as in Theorem~\ref{thm:Verdier}.
Let $Q\in C_{g}\setminus(B_F\cup\{P\})$ and let $U_Q=U\setminus F^{-1}(Q)$.  Consider the commutative diagram:
	$$
	\begin{tikzcd}
		\pi_1(F^{-1}(P))\arrow[r,"\iota_*'"]\arrow[d,"\cong"] & \pi_1(U_Q) \arrow[r, "F_*"]\arrow[d, two heads] & 
		\pi_1^{\orb}(C_{g,(1,\bar m)})\arrow[r]\arrow[d, two heads] & 1\\
		\pi_1(F^{-1}(P))\arrow[r,"\iota_*"] & \pi_1(U)\arrow[r,  "F_*"] & \pi_1^{\orb}(C_{g,(0,\bar m)})
	\end{tikzcd}
	$$
	where $\iota':F^{-1}(P)\hookrightarrow\pi_1(U_Q)$ and
	$\iota:F^{-1}(P)\hookrightarrow\pi_1(U)$ are the inclusions. Here, the vertical arrows are all induced by inclusion,
	and the top row is exact by \cite[Corollary 2.22]{ji-Eva-orbifold}. 
	The surjectivity of $F_*$ in the bottom row follows from the diagram. 
	Also, $\im(\iota_*)\subseteq\ker(F_*:\pi_1(U)\to \pi_1^{\orb}(C_{g,(0,\bar m)}))$.
	
	Let us prove the other inclusion. Since $\im(\iota_*)$ is a quotient of a normal subgroup of 
	$\pi_1(U_Q)$, it is a normal subgroup of $\pi_1(U)$. Using the presentation of $\pi_1(U_Q)$ given by
	\cite[Lemma 4.1]{ji-Eva-orbifold}, we see that $\pi_1(U_Q)/\im(\iota_*')$ has a presentation given by
	$$
	\langle \gamma_1,\ldots,\gamma_{n+2g}\mid \gamma_1^{m_1}=1,\ldots,\gamma_n^{m_n}=1\rangle
	$$
	such that
	\begin{enumerate}
		\item $F_*(\gamma_j)$ is a positively oriented meridian around a point $P_j\in C_g$, for $j=1,\ldots,n$.
		\item $\tilde \gamma:= \left(\prod_{i=1}^g[\gamma_{n+2i-1},\gamma_{n+2i}]\right)\left(\gamma_1\cdot\ldots\cdot\gamma_n\right)^{-1}$ is the class in $\pi_1(U_Q)/\im(\iota_*')$ of a positively oriented meridian around $F^{-1}(Q)$, and
		\item $F_*(\tilde\gamma)$ is a positively oriented meridian around $Q$.
	\end{enumerate}
	Thus, to get from $\pi_1(U_Q)$ to $\pi_1(U)$, one needs to quotient by the normal closure of the subgroup generated by $\tilde \gamma$. Hence, $\pi_1(U)/\im(\iota_*)$ is isomorphic to
	$$
	\langle \gamma_1,\ldots,\gamma_{n+2g}\mid \gamma_1^{m_1}=1,\ldots,\gamma_n^{m_n}=1, \left(\prod_{i=1}^g[\gamma_{n+2i-1},\gamma_{n+2i}]\right)\left(\gamma_1\cdot\ldots\cdot\gamma_n\right)^{-1}=1\rangle,
	$$
	which is isomorphic to $\pi_1^{\orb}(C_{g,(0,\bar m)})$, and by Lemma~\ref{lem:hopfian},
	$$F_*:\pi_1(U)/\im(\iota_*)\twoheadrightarrow \pi_1^{\orb}(C_{g,(0,\bar m)})$$
	is an isomorphism.
\end{proof}

\begin{rem}\label{rem:fg}
Let $F:U\to C$ be an admissible map, where $U$ is a Zariski open set of a compact K\"ahler manifold and $C$ is a smooth quasi-projective curve. As a consequence of Lemma~\ref{lem:exact}, we obtain that the kernel of the morphism induced between (orbifold) fundamental groups by an admissible map (using the maximal orbifold structure on $C$ induced by $F$) is finitely generated. Hence, the finitely generated kernel assumption on $\psi$ appearing in the Orbifold Geometric Realizability Problem~\ref{pbm:orbifoldgeometricmorphism} is necessary for the geometric realization of $\psi$ by an admissible map.
\end{rem}

\section{Main Theorem}\label{s:main}

The purpose of this section is to prove the Realizability Theorem~\ref{thm:GOPintro}. We start by proving a
generalization of Theorem~\ref{thm:summaryCatanese}. In particular, the case where $U$ is non-compact and
$G=\GG_g$ for $g\geq 2$ is not covered by Theorem~\ref{thm:summaryCatanese}. Moreover, we also reprove
the cases covered in Theorem~\ref{thm:summaryCatanese} (proved in \cite{Catanese-Fibred}) using different methods.
More concretely, we do not use the isotropic subspace theorem used in \cite{Catanese-Fibred}, and use Arapura's
Proposition~\ref{prop:Arapura} instead.

Before we begin with the statement and proofs of the main results, note that base points in the Realizability
Problems~\ref{pbm:geometricmorphism} and~\ref{pbm:orbifoldgeometricmorphism} are not needed.

\begin{rem}\label{rem:basepoint}
	Let $u,u'\in U$. Let $\psi:\pi_1(U,u)\to G$ be an epimorphism, let $h:\pi_1(U,u')\to\pi_1(U,u)$
	be a particular change-of-base-point isomorphism constructed from a path from $u'$ to $u$ and let
	$\psi':=\psi\circ h$. Suppose that $F:U\to C$ is such that $F_*:\pi_1(U,u)\to\pi_1(C,F(u))$, coincides
	with $\psi$ up to isomorphism in the target. Then, $F_*:\pi_1(U,u')\to\pi_1(C,F(u'))$ also coincides
	with $\psi'$ up to isomorphism in the target. An analogous idea works for the orbifold case considering
	$\pi_1^{\orb}(C_{g,(r,\bar m)},c)$ as a quotient of $\pi_1(C_{g,r+n},c)$, where $C_{g,r+n}$ is a smooth projective genus $g$ curve with $r+n$ punctures.
\end{rem}

\begin{thm}\label{thm:mainNoMultiple}
	Let $U$ be a Zariski open subset of a compact K\"ahler manifold.
	\begin{enumerate}
	\item\label{item:noMult1} Suppose that there exists an
	epimorphism $\psi:\pi_1(U)\to \GG_g$ with finitely generated kernel $K$, where $g\geq 2$.
	Then, there exists a smooth projective
	curve $C_g$ of genus $g$ and an admissible map $F:U\to C_g$ without multiple fibers such that
	$$
	F_*: \pi_1(U)\to \GG_g,
	$$
where $F_*$ coincides with $\psi$ up to an isomorphism of the target.
\item\label{item:noMult2} Suppose that there exists an
epimorphism $\psi:\pi_1(U)\to \FF_s$ with finitely generated kernel $K$, where $s\geq 2$. Then, there exist
$r\geq 1$ and $g\geq 0$ such that $2g+r-1=s$, $C_{g,r}$ a smooth quasi-projective curve of genus $g$ with
$r$ points removed, and an admissible map $F:U\to C_{g,r}$ without multiple fibers such that
$$
F_*: \pi_1(U)\to \FF_s,
$$
where $F_*$ coincides with $\psi$ up to an isomorphism of the target.
\end{enumerate}
Moreover, in any case, one such $F$ is unique up to isomorphism of algebraic varieties in the target.
\end{thm}

\begin{proof}
Let $G=\GG_g, \FF_s$, depending on whether we are proving the first or the second statement.
Let $l=2g$ in the first case, and $l=s$ in the second case.
Since $\psi$ is an epimorphism, the induced map
$\psi^*:\Hom(G,\CC^*)\to \Hom(\pi_1(U),\CC^*)$
is a monomorphism and hence $\psi^*$ induces an injective morphism
$\Sigma_{1}(G)\cong (\CC^*)^{l}\to \Sigma_{1}(U)$.
This implies that $\Sigma_{1}(U)$ has a positive dimensional irreducible component $W$ containing
the image of $\psi^*$. By Proposition~\ref{prop:Arapura}, there exists
an admissible map $F:U\to C_{g',r'}$ onto a curve $C_{g',r'}$ of general type, which is a smooth
quasi-projective curve of genus $g'$ with $r'$ points removed, such that $W$ is the pullback by $F^*$ of
$\Sigma_{1}(C_{g',r'})=H^1(C_{g',r'},\CC^*)$. Note that $\pi_1(C_{g',r'})$ is an infinite group
satisfying property \NINF\ by Example~\ref{exam:NINF}.

Let $C_{g',(r',\bar m)}$ be the maximal orbifold structure with respect to $F$. The original morphism $F_*$
factors through $F^{\orb}_*:\pi_1(U)\twoheadrightarrow \pi_1^{\orb}\left(C_{g',(r',\bar m)}\right)$.
Note that, abusing notation, we have previously denoted $F^{\orb}_*$ by $F_*$,
but we now use $F^{\orb}_*$ to avoid ambiguity.

Let us show that, if $F_*^{\orb}(K)$ is trivial,
then statements \eqref{item:noMult1} and \eqref{item:noMult2}
hold. Indeed, suppose that $F_*^{\orb}(K)$ is trivial. Then, $F_*^{\orb}$  factors through $\psi$.
Since the kernel $L$ of $F_*^{\orb}$ is finitely generated by Lemma~\ref{lem:exact},
Lemma~\ref{lem:NINF} implies that $F_*^{\orb}$ and $\psi$ coincide up to isomorphism in the target.
In particular, $\pi_1^{\orb}\left(C_{g',(r',\bar m)}\right)\cong G$. If $G\cong\GG_g$, then
Theorem~\ref{thm:CCOnotFreeProduct} implies that $r'=0$, $g=g'$, and that $F$ has no multiple fibers,
so $F$ and $F^{\orb}$ coincide and \eqref{item:noMult1} holds. If $G\cong \FF_s$,
Theorem~\ref{thm:CCOnotFreeProduct} implies that
$r'\geq 1$, $2g'+r'-1=s$, and that $F$ has no multiple fibers, so $F$ and $F^{\orb}$ coincide
and \eqref{item:noMult2} holds. Hence, we just need to show that $F_*^{\orb}(K)$ is trivial.

Let us consider the normal subgroup $F_*(K)$ of $\pi_1(C_{g',r'})$. Suppose that $F_*(K)$
is not trivial. Since it is finitely generated and $\pi_1(C_{g',r'})$ satisfies the \NINF\
property by Example~\ref{exam:NINF}, $F_*(K)$ has finite index in $\pi_1(C_{g',r'})$. Let us
show that this cannot happen, arguing by contradiction.

Let $\iota:K\to \pi_1(U)$ be the inclusion. Note that, since $F_*(K)$ has finite index in
$\pi_1(C_{g',r'})$ by the \NINF\ property, every non-torsion character pulls back to a non trivial
character through the inclusion $j:F_*(K)\hookrightarrow \pi_1(C_{g',r'})$.
Since $F^*:\Sigma_{1}(C_{g',r'})\to W$ is a bijection, and the image of
$\psi^{*}:\Sigma_{1}(C_{g,r})\hookrightarrow \Sigma_{1}(U)$ is positive dimensional and contained
in $W$ by the definition of $W$, there exists a non-torsion character $\xi\in\Sigma_{1}(C_{g',r'})$
such that $F^*(\xi)\in\mathrm{Im}(\psi^*)$. The pull-back of $\xi$ through $j$ and
$(F_*)_|:K\twoheadrightarrow F_*(K)$ is a non-trivial element of
$\iota^*(\mathrm{Im}(\psi^*))\subset \iota^*(W)$, so $\iota^*(\mathrm{Im}(\psi^*))$ is not
trivial. However, since  $\psi\circ\iota$ is trivial, we get that $\iota^*(\mathrm{Im}(\psi^*))$
is trivial, which is a contradiction. Hence, $F_*(K)$ is trivial.

Let us use that $F_*(K)$ is trivial to prove that $F^{\orb}_*(K)$ is trivial.
Note that $F^{\orb}_*(K)\trianglelefteq \pi_1^{\orb}\left(C_{g',(r',\bar m)}\right)$ is contained in the kernel
of $\pi_1^{\orb}\left(C_{g',(r',\bar m)}\right)\twoheadrightarrow \pi_1(C_{g',r'})$, so
$F^{\orb}_*(K)$ has infinite index as a subgroup of $\pi_1^{\orb}\left(C_{g',(r',\bar m)}\right)$.
Since $F^{\orb}_*(K)$ is finitely generated and $\pi_1^{\orb}\left(C_{g',(r',\bar m)}\right)$ satisfies
property \NINF\ by Lemma~\ref{lem:FgroupNINF}, this implies that $F^{\orb}_*(K)$ is trivial.

The ``moreover'' part of the statement follows from Proposition~\ref{prop:Arapura}.
\end{proof}

The strategy to prove Theorem~\ref{thm:GOPintro} is as follows: from an epimorphism $\psi$ as in the
Orbifold Geometric Realizability Problem~\ref{pbm:orbifoldgeometricmorphism}, we build a finite cover
$\widetilde U$ of $U$ that we can apply Theorem~\ref{thm:mainNoMultiple} to and get an admissible map
$f:\widetilde U\to D$ to a quasi-projective curve $D$ with no multiple fibers. Then, we will use $f$
to find the desired $F:U\to C$ as in the Orbifold Geometric Realizability Problem~\ref{pbm:orbifoldgeometricmorphism}.
The following result describes how to obtain such a map $f$,
and provides the main tools that we will use for the construction of~$F$.

\begin{lemma}\label{lem:cover}
Let $U$ be a Zariski open set in a compact K\"ahler manifold, and let $\psi:\pi_1(U)\to\GG_{g,(r,\bar m)}$ be
an epimorphism such that its kernel $K$ is finitely generated, with $\chi_{g,(r,\bar m)}<0$. Let $G$ be a finite index normal
subgroup of $\GG_{g,(r,\bar m)}$ which contains no torsion elements other than the identity, as given by
Lemma~\ref{lem:Fenchel}, and let $p:\widetilde U\to U$ be the finite regular covering map corresponding to the
normal subgroup $\psi^{-1}(G)$ of $\pi_1(U)$. Then,
\begin{enumerate}
	\item\label{item:exist}
	There exists a smooth quasi-projective curve $D$ of general type and an admissible map $f:\widetilde U \to D$ with no multiple
	fibers such that
	$$
	f_*:\pi_1(\widetilde U)\to \pi_1(D)
	$$
	coincides with $\psi_|\circ p_*:\pi_1(\widetilde U)\to G$ up to isomorphism in the target. Moreover, $p$ induces an isomorphism between $\ker f_*$ and $K$.
	\item\label{item:deckcharacteristic} For every deck transformation $d:\widetilde U\to\widetilde U$ of the
	covering map $p$, the maps
	$$
	(f\circ d)^*:\Sigma_1(D)\to\Sigma_1(\widetilde U)
	$$
	and
	$$
	f^*:\Sigma_1(D)\to\Sigma_1(\widetilde U)
	$$
	have the same image.
	\item\label{item:algebraic} For every deck transformation $d:\widetilde U\to\widetilde U$ of the covering map
	$p$, there exists a unique algebraic isomorphism $d':D\to D$ such that
	$$
	f\circ d=d'\circ f.
	$$
	Hence, the deck transformation group of $p$ induces a compatible action on $D$ by algebraic isomorphisms.
	\end{enumerate}
\end{lemma}

\begin{proof}
	
	We start by fixing a base point $u\in U$, and consider $\psi$ as an epimorphism from $\pi_1(U,u)\to \GG_{g,(r,\bar m)}$. Let $\widetilde u\in p^{-1}(u)\subset\widetilde U$.
	
	By Theorem~\ref{thm:finitecover}, $\widetilde U$ is also a Zariski open set in a compact K\"ahler manifold. Also note that the kernel of $\psi_|\circ p_*$ is isomorphic to $K\cap \psi^{-1}(G)=K$, so it is finitely generated. 	Lemma~\ref{lem:Fenchel} implies that $G$ satisfies that $G\cong \GG_{g'}$ for some $g'\geq 2$ if $r=0$ or  $G\cong \FF_s$ for some $s\geq 2$ if $r\geq 1$. Part~\eqref{item:exist} follows from Theorem~\ref{thm:mainNoMultiple} applied to $\psi_|\circ p_*$.
	
	Let $f$ be as in part~\eqref{item:exist}, and let $\varphi:\pi_1(D,f(\widetilde u))\to \GG_{g,(r,\bar m)}$
	be the inclusion onto $G$ such that $\varphi\circ f_*=\psi\circ p_*$. Let $C_{g,r}$ be a smooth quasi-projective
	curve of genus $g$ with $r$ points removed, let $c\in C_{g,r}$ be any point, let $C_{g,(r,\bar m)}$ the orbifold
	structure on $C_{g,r}$ obtained from choosing $n$ points (all different from $c$) with multiplicities
	$m_1,\ldots,m_n$, and let $\phi:\GG_{g,(r,\bar m)}\to \pi_1^{\orb}\left(C_{g,(r,\bar m)},c\right)$ be an isomorphism.
	By Lemma~\ref{lem:Fenchel}, the Galois cover of $p':D'\to C_{g,r}$ corresponding to the subgroup $\phi(G)$ satisfies
	that the branch divisor is $n$ points with multiplicities $m_1,\ldots,m_n$. Let $x\in (p')^{-1}(c)$. Note that $\phi$
	induces an isomorphism $\phi_|:\pi_1(D,f(\widetilde u))\to \pi_1(D',x)$ such that $\phi\circ\varphi=(p')_*\circ \phi_|$.
	In other words, we have the following commutative diagram:
	$$
	\begin{tikzcd}
	\pi_1(\widetilde U,\widetilde u)\arrow[r,two heads,"f_*"]\arrow[d, hook, "p_*"] & \pi_1(D,f(\widetilde u))\arrow[r,"\phi_|","\cong"']\arrow[d,hook,"\varphi"] & \pi_1(D',x)\arrow[d,hook,"(p')_*"]\\
	\pi_1(U,u) \arrow[r, two heads,"\psi"]& \GG_{g,(r,\bar m)} \arrow[r,"\phi","\cong"'] & \pi_1^{\orb}\left(C_{g,(r,\bar m)},c\right)
	\end{tikzcd}
	$$
	
	Let us show part~\eqref{item:deckcharacteristic}. Let $d:\widetilde U\to\widetilde U$ be the deck
	transformation corresponding to a certain element $[\gamma]$ of
	$\pi_1(U,u)/p_*\left(\pi_1(\widetilde U,\widetilde u)\right)$, where $\gamma$ is a loop in $U$ starting at $u$. That is, $\widetilde\gamma_{\widetilde u}(1)=d(\widetilde u)$, where $\widetilde\gamma_{\widetilde u}$ is the lift of $\gamma$ starting at $\widetilde u$. Note that
	$$\pi_1(U,u)/p_*\left(\pi_1(\widetilde U,\widetilde u)\right)\overset{\psi}{\cong}\GG_{g,(r,\bar m)}/G\overset{\phi}{\cong}\pi_1^{\orb}\left(C_{g,(r,\bar m)},c\right)/(p')_*(\pi_1(D',x))$$
	Let $d_2:D'\to D'$ be the covering transformation corresponding to
	$d\in \pi_1(U,u)/p_*\left(\pi_1(\widetilde U,\widetilde u)\right)$ through $\phi\circ\psi$, and let $[\delta]:=\phi\circ\psi([\gamma])\in \pi_1^{\orb}\left(C_{g,(r,\bar m)},c\right)$.
	Let $M_+\subset C_{g,r}$ be the set of orbifold points of multiplicity $\geq 2$. Note that, since the natural
	projection $\pi_1(C_{g,r}\setminus M_+, c)\twoheadrightarrow \pi_1^{\orb}(C_{g,r}, c)$
	is surjective, we may assume that $\delta$ is a loop in $C_{g,r}\setminus M_+$ starting at $c$.
	
	Consider the following commutative diagram:
	\begin{equation}\label{eq:commbase}
	\begin{tikzcd}
		\pi_1(U, u)\arrow[r,"\mathrm{Id}", "\cong"']\arrow[d,two heads, "\phi\circ\psi"]&\pi_1(U,u)\arrow[r,"\gamma_{\#}", "\cong"']&\pi_1(U,u)\arrow[d,two heads, "\phi\circ\psi"]\\
		\pi_1^{\orb}(C_{g,r}, c)\arrow[r,"\mathrm{Id}", "\cong"']& \pi_1^{\orb}(C_{g,r}, c)\arrow[r,"\delta_{\#}", "\cong"']&\pi_1^{\orb}(C_{g,r}, c),
	\end{tikzcd}
	\end{equation}
	where $\gamma_{\#}$ (resp. $\delta_{\#}$) denotes the conjugation map $[\alpha]\mapsto [\gamma]\cdot[\alpha]\cdot[\gamma]^{-1}$ (resp. $[\beta]\mapsto [\delta]\cdot[\beta]\cdot[\delta]^{-1}$).
	
	Let $\mathcal Q:=(p')^{-1}(M_+)$. Let $\widetilde \delta_x$ be the lift of $\delta$ by $p'$ starting at $x$, which exists and is unique because $p'_|:D'\setminus\mathcal Q\to C_{g,r}\setminus M_+$ is an unramified cover. Interpreting the fundamental groups of the covering spaces and as subgroups of the fundamental group of the bases through the monomorphism induced on fundamental groups by the corresponding covering maps, the previous commutative diagram determines the following unique commutative diagram between the corresponding subgroups:
	\begin{equation}\label{eq:commtop}
		\begin{tikzcd}
		\pi_1(\widetilde{U},\widetilde u)\arrow[r,"d_*", "\cong"']\arrow[d, two heads, "\phi_|\circ f_*"]&\pi_1(\widetilde{U},d(\widetilde u))\arrow[r,"\left(\widetilde{\gamma}_{\widetilde u}\right)_{\#}", "\cong"']&\pi_1(\widetilde{U},\widetilde u)\arrow[d, two heads, "\phi_|\circ f_*"]\\
		\pi_1(D', x)\arrow[r,"(d_2)_*", "\cong"']& \pi_1(D', d_2(x))\arrow[r,"\left(\widetilde{\delta}_x\right)_{\#}", "\cong"']&\pi_1(D', x),
	\end{tikzcd}
	\end{equation}
	where $\left(\widetilde{\gamma}_{\widetilde u}\right)_{\#}$ (resp. $\left(\widetilde{\delta}_x\right)_{\#}$) denotes the  map $[\alpha']\mapsto \left[\widetilde{\gamma}_{\widetilde u}*\alpha'*\overline{\left(\widetilde{\gamma}_{\widetilde u}\right)}\right]$ (resp. $[\beta']\mapsto \left[\widetilde{\delta}_x*\beta'*\overline{\left(\widetilde{\delta}_x\right)}\right]$). Indeed, each of the rows of \eqref{eq:commbase} lift uniquely to the rows of \eqref{eq:commtop}, so \eqref{eq:commtop} must commute.
	Now, we consider the natural (abelianization) morphisms from the fundamental group to the first homology with integer coefficients. The previous commutative diagram induces the following commutative diagram in homology: 
	$$
	\begin{tikzcd}
		H_1(\widetilde{U},\ZZ)\arrow[r,"d_*", "\cong"']\arrow[d,two heads, "\mathrm{Ab}(\phi_|)\circ f_*"]&H_1(\widetilde{U},\ZZ)\arrow[r,"\mathrm{Id}", "\cong"']&H_1(\widetilde{U},\ZZ)\arrow[d,two heads, "\mathrm{Ab}(\phi_|)\circ f_*"]\\
		H_1(D', \ZZ)\arrow[r,"(d_2)_*", "\cong"']& H_1(D', \ZZ)\arrow[r,"\mathrm{Id}", "\cong"']&H_1(D', \ZZ)
	\end{tikzcd}
	$$
	This implies that
	$$
	d^*\circ f^* \circ (\phi_|)^*,f^* \circ (\phi_|)^*\circ (d_2)^*:\Hom_{\ZZ}(H_1(D',\ZZ),\CC^*)\to\Hom_{\ZZ}(H_1(\widetilde U,\ZZ),\CC^*)
	$$
	coincide, so both induce the same morphism at the level of characteristic varieties. In particular, since $(d_2)^*:\Sigma_1(D')\to\Sigma_1(D')$ and $(\phi_|)^*:\Sigma_1(D')\to\Sigma_1(D)$ are isomorphisms, $f^*$ and $d^*\circ f^* $ (seen as morphisms between first characteristic varieties) have the same image. This concludes the proof of part~\eqref{item:deckcharacteristic}.
	
	Note that $f\circ d$, like $f$, is an admissible map. Part~\eqref{item:algebraic} follows from part~\eqref{item:deckcharacteristic} and Proposition~\ref{prop:Arapura}.
\end{proof}

 The following result allows us to lift admissible maps to suitable branched covers.
 
 \begin{lemma}\label{lem:Steincover}
	Let $U$ be a Zariski open set inside of a compact K\"ahler manifold, and let $F:U\to C_{g,r}$ be an admissible map
	to a smooth quasi-projective curve $C_{g,r}$. Let $C_{g,(r,\bar m)}$ be the maximal orbifold structure on $C_{g,r}$
	with respect to $F$. Suppose that $\chi_{g,(r,\bar m)}<0$. Let $N\trianglelefteq \pi_1^{\orb}(C_{g,(r,\bar m)})$
	be a finite index normal subgroup with no non-trivial torsion elements, whose existence is guaranteed by
	Lemma~\ref{lem:Fenchel}. Let $p:\widetilde U\to U$ be the regular (unramified) cover associated to the
	subgroup $(F_*)^{-1}(N)$, and let $q:D\to C_{g,r}$ be the (Galois) branched cover corresponding to $N$ as in
	Lemma~\ref{lem:Fenchel}. Then, $F$ lifts to an admissible map $f:\widetilde U\to D$ with no multiple fibers
	such that the following diagram commutes:
	\begin{equation}\label{eq:cover}
	\begin{tikzcd}
	\pi_1(\widetilde U)\arrow[r, two heads, "f_*"]\arrow[d, hook , "p_*"]& \pi_1(D)\arrow[d,hook,"q_*"]\\
	\pi_1(U)\arrow[r, two heads, "F_*"] & \pi_1^{\orb}(C_{g,(r,\bar m)}).
	\end{tikzcd}
	\end{equation}
\end{lemma}

\begin{proof}
Let $M_+\subset C_{g,r}$ the set of $n$ points with multiplicities in the orbifold structure $C_{g,(r,\bar m)}$.
We use facts about branched covers which can be found in \cite[Section 1]{Uludag}.
Consider $\overline F:X\to C_g$ be a holomorphic enlargement of $F$, where $X$ is a compact K\"ahler manifold such
that $X\setminus U$ is a simple normal crossings divisor. Note that $q$ also extends to a branched cover of smooth
compact projective curves $\overline q:\overline D\to C_g$ such that $D=\overline q^{-1}(C_{g,r})$. Consider the
following pullback diagram:
$$
\begin{tikzcd}
Y\arrow[r, "\overline f"]\arrow[dr, phantom, "\lrcorner", very near start] \arrow[d, "\overline \alpha"] &
\overline D\arrow[d, "\overline q"]\\
X\arrow[r, "\overline F"]& C_g
\end{tikzcd}
$$
Note that $\overline\alpha$ is a proper complex analytic morphism with finite fibers between compact complex spaces,
and that $\overline f$ has connected fibers. Let $\eta:Y'\to Y$ be the normalization of $Y$. Then,
$\overline\alpha\circ\eta$ is a proper surjective complex analytic morphism with finite fibers between two normal
compact complex spaces, so it is a branched cover. Note that, since the normalization morphism is a biholomorphic
over the set of normal points \cite[Chapter 8, \S 4]{GR-sheaves} and
$q_|:D\setminus q^{-1}(M_+)\to C_{g,r}\setminus M_+$ is a regular cover, then
$$
\overline\alpha\circ\eta_|:(\overline\alpha\circ\eta)^{-1}(\overline F^{-1}(C_{g,r}\setminus M_+))\to
\overline F^{-1}(C_{g,r}\setminus M_+)
$$
is a regular cover, so, further restricting the base to $U\setminus F^{-1}(M_+)$, we get that
$$
\overline\alpha\circ\eta_|:(\overline\alpha\circ\eta)^{-1}(U\setminus F^{-1}(M_+))\to  U\setminus F^{-1}(M_+)
$$
is a regular cover. In fact, this is precisely the regular cover of $U\setminus F^{-1}(M_+)$ obtained by the pullback diagram
$$
\begin{tikzcd}
V\arrow[r]\arrow[dr, phantom, "\lrcorner", very near start] \arrow[d] &  D\arrow[d, "q"]\\
U\arrow[r, "F"]& C_{g,r}
\end{tikzcd}
$$
By construction, this regular cover is equivalent to $p_|:\widetilde U\setminus p^{-1}(F^{-1}(M_+))\to U\setminus F^{-1}(M_+)$. By the uniqueness in the Grauert-Remmert theorem~\cite{GrauertRemmert} (as stated in \cite[Theorem 1]{NambaGerms}), $p$ coincides with $\overline\alpha\circ\eta$ over $U$, so the lift $f$ of $F$ exists and has connected generic fibers. By Theorem~\ref{thm:finitecover} and Remark~\ref{rem:SteinAdmissible}, $f$ is an admissible map from a Zariski open set in a compact K\"ahler manifold to $D$, which, by Lemma~\ref{lem:Fenchel}, is a curve of general type.

Note that the maximal orbifold structure on $C_{g,r}$ with respect to $q$ is also $C_{g,(r,\bar m)}$, so the map $q_*$
in diagram \eqref{eq:cover} is well defined and injective by \cite[Lemma 1.3]{Uludag}. Also, diagram \eqref{eq:cover}
necessarily commutes by how the induced morphisms between (orbifold) fundamental groups are defined. Since the image of
$p_*$ contains $\ker F_*$ by construction, $\ker f_*$ is finitely generated. Note that, since $D$ is of general type,
the curve orbifold group $\pi_1^{\orb}(D_{\bar m'})$ has negative Euler characteristic for every orbifold structure
$\bar m'$ on $D$, and in particular it satisfies property \NINF\ by Lemma~\ref{lem:FgroupNINF}. In particular, if
$\bar m'$ is maximal with respect to $f$, then $f_*$ factors through $\pi_1^{\orb}(D_{\bar m'})$, and
Lemma~\ref{lem:NINF} implies that $\pi_1^{\orb}(D_{\bar m'})\cong\pi_1(D)$. Hence, by
Theorem~\ref{thm:CCOnotFreeProduct}, $f$ does not have multiple fibers.
\end{proof}

Now, we are ready to prove Theorem~\ref{thm:GOPintro}

\begin{thm}\label{thm:GOP}
	Let $U$ be a Zariski open set inside of a compact K\"ahler manifold. Suppose that there exists an epimorphism
	$\psi:\pi_1(U)\to \GG_{g,(r,\bar m)}$ with finitely generated kernel $K$, where $\chi_{g,(r,\bar m)}<0$. Then,
	there exists a smooth quasi-projective curve $C_{g',r'}$ and an admissible map $F:U\to C_{g',r'}$ such that the
	following hold:
	\begin{enumerate}
	\item \label{item:mult}
$F$ induces an orbifold morphism
$$
F:U\to C_{g',(r',\bar m)},
$$
where $C_{g',(r',\bar m)}$ is maximal with respect to $F$.
\item \label{item:psi}
$F_*:\pi_1(U)\to \pi_1^{\orb}(C_{g',(r',\bar m)})$ coincides with $\psi$ up to isomorphism in the target.
	\item\label{item:proj} $C_{g',r'}$ is projective if and only if $r=0$.
	\item\label{item:r0} If $r=0$, then $g=g'$. If $r\geq 1$, then $2g'+r'=2g+r$.
	\end{enumerate}
	
	Moreover, one such $F$ is unique up to isomorphism of algebraic varieties in the target.
\end{thm}

\begin{proof}
Let $p:\widetilde U\to U$, $D$ and $f:\widetilde U\to D$ as in Lemma~\ref{lem:cover}.
	
Since $D$ is a curve, it has a cover by affine open sets which are invariant under the action of $\mathrm{Deck}(p:\widetilde U\to U)$ from Lemma~\ref{lem:cover}, part~\eqref{item:algebraic} (the complement of the orbits of a finite number of points). Thus, the quotient $\theta:D\to C'$ map by this action exists as an algebraic morphism, where $C'$ is a 
curve (hence quasi-projective, but not necessarily smooth).
The map $f:\widetilde U\to D$ induces a complex analytic morphism $f':U\to C'$ between the quotients. By the universal property of
the normalization \cite[Chapter 8, \S 4]{GR-sheaves}, $f'$ lifts to $F':U\to S$, where $S$ is the
normalization of $C'$. In particular, $S$ is a smooth quasi-projective curve. Applying Stein factorization to a holomorphic enlargement of  $F':U\to S$ as in Remark~\ref{rem:SteinAdmissible}, we obtain that $F'=g \circ F$, where $C$ is a smooth quasi-projective curve,
$F:U\to C$ is admissible, and $g:C\to S$ is a finite morphism.

Since the normalization $S\to  C'$ is a birational equivalence, a generic fiber of $f'$
is also a generic fiber of $F'$, which is a disjoint union of generic fibers of $F$.
Moreover, since the generic fiber of $\theta$ is finite, the preimage through $p$
of a generic fiber of $f'$ is a disjoint union of generic fibers of $f$. Restricting to connected
components, one finds $P\in C\setminus B_{F}$ and $P'\in D\setminus{B_{f}}$ such that
$p:f^{-1}(P')\to F^{-1}(P)$ is a finite covering map, where $B_F$ and $B_{f}$ are the sets of atypical values of the corresponding maps.

Let us check that the morphism $F_*:\pi_1(U)\to \pi_1^{\orb}(C)$ coincides with $\psi$ up to
isomorphism in the target, where $C$ is endowed with the maximal orbifold structure with respect
to $F:U\to C$. Let $L$ be the kernel of $F_*:\pi_1(U)\to \pi_1^{\orb}(C)$. Consider the commutative diagram
$$
\begin{tikzcd}
	\pi_1(f^{-1}(P'))\arrow[r, "\iota_*"]\arrow[d, "p_*"] & \pi_1(\widetilde U)\arrow[d, "p_*"]&\\
	\pi_1(F^{-1}(P))\arrow[r, "i_*"]&\pi_1(U)\arrow[r, "F_*"]& \pi_1^{\orb}(C),
\end{tikzcd}
$$
where the horizontal arrows are induced by the inclusions. The vertical arrow on the left is
a finite covering space, so the image of $p_*\circ\iota_*$ (which is equal to $K$ by Lemma~\ref{lem:exact} and Lemma~\ref{lem:cover}\eqref{item:exist}) is a
finite index subgroup of the image of $i_*$ (which is equal to $L$ by Lemma~\ref{lem:exact}). Thus,
$$
\pi_1^{\orb}(C)\cong \pi_1(U)/L\cong (\pi_1(U)/K)/(L/K)
$$
is the quotient of an infinite group $\pi_1(U)/K\cong \GG_{g,(r,\bar m)}$ by a finite normal subgroup $L/K$ (so in particular, $\pi_1^{\orb}(C)$ is an infinite group). However, by Remark~\ref{rem:propertieshyperbolic} part~\eqref{part:finiteNormal} if $r=0$ and the Kurosh subgroup theorem if $r\geq 1$, $\GG_{g,(r,\bar m)}$ does not have any non-trivial finite normal subgroups. Thus, $L=K$, and, by Lemma~\ref{lem:NINF}, $F_*$ coincides with $\psi$ up to isomorphism in the target. Together with Theorem~\ref{thm:CCOnotFreeProduct}, this proves parts \eqref{item:mult} and \eqref{item:psi}. Parts \eqref{item:proj} and \eqref{item:r0} also follow from Theorem~\ref{thm:CCOnotFreeProduct}.

Let us prove the ``moreover'' part of the statement, regarding the uniqueness of $F$ up to algebraic isomorphism in
the target. Let $F_i:U\to C_i$ be admissible maps satisfying parts~\eqref{item:mult} to~\eqref{item:r0} for $i=1,2$. $(F_i)_*$ coincides with $\psi$ up to isomorphism in the target, so let $N_i\trianglelefteq \pi_1^{\orb}\left((C_i)_{\bar m}\right)$ be the normal subgroups corresponding to those identifications to $G\trianglelefteq \GG_{g,(r,\bar m)}$ as in Lemma~\ref{lem:cover}, which was the one used to construct $p:\widetilde U\to U$. Let $q_i:D_i\to C$ be the branched cover associated to $N_i$ as in Lemma~\ref{lem:Fenchel}, so $D_i$ is of general type. By Lemma~\ref{lem:Steincover}, $F_i$ lifts to an admissible map with no multiple fibers $f_i:\widetilde U\to D_i$. The commutativity of diagram~\eqref{eq:cover} in  Lemma~\ref{lem:Steincover} and the uniqueness in Theorem~\ref{thm:mainNoMultiple} imply that there exists an algebraic isomorphism $h:D_1\to D_2$ such that
\begin{equation}\label{eq:factor}
f_2=h\circ f_1
\end{equation}Moreover, by Lemma~\ref{lem:cover}\eqref{item:algebraic} and the fact that $(F_i)_*$ identifies the Galois groups of $p$ and $q_i:D_i\to C$ by the commutativity of diagram~\eqref{eq:cover} in Lemma~\ref{lem:Steincover}, equation~\eqref{eq:factor} descends to the quotients by the Galois groups as
$$
F_2=H\circ F_1,
$$
where $H:C_1\to C_2$ is an isomorphism.
\end{proof}

\section{The Orbifold Geometric Realizability Problem for groups with $\chi_{g,(r,\bar m)}\geq 0$}\label{s:sharp}
The goal of this section is to show that the Orbifold Geometric Realizability Problem~\ref{pbm:orbifoldgeometricmorphism} does not have a
positive answer for curve orbifold groups $\GG_{g,(r,\bar m)}$ with $\chi_{g,(r,\bar m)}\geq 0$ and
thus Theorem~\ref{thm:GOP} is sharp. We will do this by constructing examples for each one of those groups. 
By Remarks~\ref{rem:freeProduct} and~\ref{rem:infinite}, those groups are:
\begin{itemize}
 \item The Euclidean (resp. spherical) \cco\ groups (with $\chi_{g,\bar m}=0$, resp. $\chi_{g,\bar m}>0$)
 from List~\ref{list}.
 \item $\ZZ$ and $\ZZ_2*\ZZ_2$,
\end{itemize}
Note that finite cyclic groups are both open and (spherical) compact curve orbifold groups.

The following example shows that Theorem~\ref{thm:GOP} cannot be extended to any of the
Euclidean \cco\ groups. In it, we find quasi-projective varieties $U$ such that $\pi_1(U)$ is a
Euclidean \cco\ group however, $U$ does not admit an admissible map $F:U\to S$ onto a smooth projective
curve such that $\pi_1(U)\cong\pi^{\orb}_1(S_{\bar m})$ for some orbifold structure given by the multiple
fibers of~$F$. 

\begin{exam}[Euclidean  \cco\ groups]
For the case $\GG_1\cong\ZZ^2$, Hodge theory provides an obstruction: Let $U=(\CC^*)^2$, which is an
affine variety such that $\pi_1(U)\cong\ZZ^2$. Note that $H^1((\CC^*)^2,\QQ)$ is pure of type $(1,1)$, so in
particular, it has weight~$2$, whereas the first
cohomology of a smooth projective curve of genus $1$ has weight $1$. Hence, there does not exist any surjective algebraic morphism $F:U\to S$ inducing an isomorphism at the level of fundamental groups, where $S$ is a smooth projective curve of genus $1$.

Let us deal with the remaining four cases. Consider the following free and properly discontinuous action
$\sigma_k$ of $\ZZ_k$ by algebraic automorphisms on $(\CC^*)^2$. The points $\cP_k\subset (\CC^*)^2$
are those with non-trivial isotropy group.
\begin{enumerate}
\item $\sigma_2(s,t)=\left(\frac{1}{s},\frac{1}{t}\right)$ and $\cP_2:=\{(\pm 1,\pm 1),(\pm 1,\mp 1)\}$,
\item $\sigma_3(s,t)=\left(\frac{1}{t},\frac{s}{t}\right)$ and $\cP_3:=\{(1,1),(w,\bar w), (\bar w,w)\}$,
\item $\sigma_4(s,t)=\left(\frac{1}{t},s\right)$ and $\cP_4:=\cP_2$,
\item $\sigma_6(s,t)=\left(st,\frac{1}{s}\right)$ and $\cP_6:=\cP_2\cup\{(w,w), (\bar w,\bar w)\}$,
\end{enumerate}
where $w^2+w+1=0$. Consider $\bar\pi_k:(\CC^*)^2\to \bar U_k$ the corresponding quotient.
Since $(\CC^*)^2$ is affine, the restriction $\pi_k:V_k\to U_k$, $V_k=(\CC^*)^2\setminus \cP_k$ produces an
unbranched cover over a smooth quasi-projective variety.
\begin{enumerate}
\item $\bar U_2=\{(x,y,z)\in\CC^3\mid x^2+y^2+z^2-xyz=4\}$,
$\bar\pi_2(s,t)=(\frac{s^2+1}{s},\frac{t^2+1}{t},\frac{s^2t^2+1}{st})$, and
$U_2=\bar U_2\setminus \bar\pi_2(\cP_2)$.
\item $\bar U_3=\{(x,y,z)\in\CC^3\mid x^3+y^3+z^2-xyz-6xy+3z+9=0\}$,
$\bar\pi_3(s,t)=(\frac{s^2t+s+t^2}{st},\frac{st^2+t+s^2}{st},\frac{s^3t^3+s^3+t^3}{s^2t^2})$, and
$U_3=\bar U_3\setminus \bar\pi_3(\cP_3)$.
\item $\bar U_4=\{(x,y,z)\in\CC^3\mid x^4-7x^2y+y^3-xyz-3x^2+8y^2+2xz+z^2+16y=0\}$,
$\bar\pi_4(s,t)=\left(\frac{(st+1)(s+t)}{st},\frac{(s^2+1)(t^2+1)}{st},\frac{(st^3+1)(s^3+t)}{s^2t^2}\right)$,
and $U_4=\bar U_4\setminus \bar\pi_4(\cP_4)$.
\item
$\bar U_6=\{(x,y,z)\in\CC^3\mid f_6(x,y,z)=0\}$, where
$$
{\scriptstyle{
f_6 = x^{5} + x^{4} - 8 x^{3} y - 23 x^{3} - 9 x^{2} y + 14 x y^{2} + y^{3}
+ 2 x^{2} z - x y z - 20 x^{2} + 82 x y + 31 y^{2} - 2 x z - 4 y z + z^{2} + 120 x + 132 y - 12 z + 144,
}}
$$
$$
\bar\pi_6(s,t)=
{\scriptscriptstyle{
\left(
\frac{s^{2} t^{2} + s^{2} t + s t^{2} + s + t + 1}{st},
\frac{s^{4} t^{3} + s^{3} t^{4} + s^{3} t + s t^{3} + s + t}{s^2t^2},
\frac{s^{4} t^{4} + s^{4} t^{2} + s^{2} t^{4} + s^{2} + t^{2} + 1}{s^2t^2}\right),
}}
$$
and $U_6=\bar U_6\setminus \bar\pi_6(\cP_6)$.
\end{enumerate}
Note that $\pi_1(V_k)\cong\ZZ^2$ and $\Deck(\pi_k)=\ZZ_k$.
On the other hand, one has the following exact sequence:
\begin{equation}\label{eq:ses}
1\to\ZZ^2\to G_k\ \rightmap{\varphi_k}\ \ZZ_k\to 1,
\end{equation}
where $G_2=\GG_{0,(2,2,2,2)}$, $G_3=\GG_{0,(3,3,3)}$, $G_4=\GG_{0,(2,4,4)}$,
$G_6=\GG_{0,(2,3,6)}$, and $\varphi_k(x_i)=\frac{k}{m_i}$.
For each of these, it is straightforward to describe $\ZZ^2$ as a normal subgroup of $G_k$ as:
\begin{enumerate}
\item $\ZZ (x_2x_1)\times \ZZ (x_3x_1)\trianglelefteq G_2$,
\item $\ZZ (x_2x_1^2)\times \ZZ (x_1x_2x_1)\trianglelefteq G_3$.
\item $\ZZ (x_1x_2^2)\times \ZZ (x_2x_1x_2)\trianglelefteq G_4$.
\item $\ZZ( [x_1,x_2])\times \ZZ (x_2[x_1,x_2]x_2^{-1})\trianglelefteq G_6$.
\end{enumerate}
In fact, the short exact sequence~\eqref{eq:ses} splits because $x_3$ has order $k$ in $G_k$
by Remark~\ref{rem:propertiesEuclidean} part~\eqref{part:orderE}. Hence, the abelian extension $G_k$
is determined by the corresponding
action $h_k$ of $\ZZ_k$ on $\ZZ^2=\ZZ a\times \ZZ b$. Note that
$h_2(a,b)=(-a,-b)$, $h_3(a,b)=(-b,a-b)$, $h_4(a,b)=(-b,a)$, and $h_6(a,b)=(a+b,-a)$.
This coincides with the action by deck transformations of $\ZZ_k$ on $\pi_1(V_k)$.
Moreover, the ramified covering $\tilde \pi_k:(\CC^*)^2\to \bar U_k$ totally ramifies at
$(1,1)\in(\CC^*)^2$, since $\bar \pi_k^{-1}((1,1))=\{(1,1)\}$.
This implies that the short exact sequence
\begin{equation}\label{eq:sespi1}
1\to\pi_1(V_k)\xrightarrow{(\pi_k)_*}\pi_1(U_k)\to \ZZ_k\to 1
\end{equation}
splits, and hence~$\pi_1(U_k)\cong G_k$.

Let us see that there is no algebraic morphism $F:U_k\to S$ such that
$F_*:\pi_1(U_k)\to\pi_1^{\orb}(S_{\bar m})$
is an isomorphism, where $S$ and $\bar m$ are as in Theorem~\ref{thm:GOP}.
Since $H_1(U_k;\QQ)=0$, one would have that $S=\PP^1$. Moreover, since $G_k$ is not hyperbolic,
the orbifold structure on $\PP^1$ would have to be $(2,2,2,2)$, $(3,3,3)$, $(2,4,4)$, and $(2,3,6)$,
for $k=2,3,4$, and $6$ respectively, since those are the only non-hyperbolic \cco\ infinite groups
with the prescribed abelianizations. By the proof of Lemma~\ref{lem:Fenchel}, there would exist a
finite Galois cover $E_k\to\PP^1_{\bar m}$ associated with the kernel of $\pi_1(U_k)\to\ZZ_k$ such
that the branching divisor is given by the tuple $\bar m$ and $E_k$ is an elliptic curve. Since the
action is equivariant, there would exist an induced rational map $\PP^2\dashrightarrow E_k$, which
yields a contradiction since $E_k$ is not simply connected.
\end{exam}

The open curve orbifold groups with $\chi_{g,(r,\bar m)}=0$ are $\ZZ$ and $\ZZ_2*\ZZ_2$. The following two examples shows that Theorem~\ref{thm:GOP} cannot be extended to these groups either.

\begin{exam}[$\ZZ$]
Let $E$ be an elliptic curve, so $\pi_1(E)\cong\ZZ^2$, and let $\psi:\pi_1(E)\to\ZZ$ be an epimorphism.
Hence, $\ker \psi\cong \ZZ$. However, since $E$ is compact and $\ZZ$ is not a \cco\ group, $\psi$ cannot
be realized by an algebraic morphism from $E$ to a curve.
\end{exam}

\begin{exam}[$\ZZ_2*\ZZ_2$]
By \cite[Theorem II.2.3]{FriedmanMorgan}, there exists a smooth projective variety $X$ such that
$\pi_1(X)\cong\GG_{0,(2,2,2,2)}$. Consider the group $\ZZ_2*\ZZ_2$ and denote by $\alpha$ a generator
of the first free factor, and $\beta$ a generator of the second free factor. Consider the presentation
of $\GG_{0,(2,2,2,2)}$ of equation~\eqref{eq:Ggm}. The rule $x_1,x_2\mapsto \alpha$, $x_3,x_4\mapsto \beta$
defines an epimorphism $\psi:\GG_{0,(2,2,2,2)}\to \ZZ_2*\ZZ_2$ which cannot be realized by an algebraic
map $F$ from $X$ to a curve with an extra orbifold structure, since $X$ is compact and $\ZZ_2*\ZZ_2$
is not a \cco\ group. However, the kernel of $\psi$ is the cyclic group generated by $x_1x_2$, which is
finitely generated.
\end{exam}

We have shown that the Orbifold Geometric Realizability Problem~\ref{pbm:orbifoldgeometricmorphism} does not have a positive answer for curve orbifold groups $\GG_{g,(r,\bar m)}$ with $\chi_{g,(r,\bar m)}=0$, which, by Remark~\ref{rem:infinite}, are all infinite groups. Now, let us show it for the groups with $\chi_{g,(r,\bar m)}>0$. These are the spherical \cco\ groups appearing in List~\ref{list}, which are all finite.

\begin{exam}[Finite cyclic groups]
	Let $E$ be an elliptic curve, so $\pi_1(E)\cong\ZZ^2$. Consider any epimorphism $\psi:\pi_1(E)\to\ZZ_k$, for $k\geq 1$. Suppose that $\psi$ is realized geometrically. Finite cyclic groups can be both open and compact curve orbifold groups, but since $E$ is projective, then a geometric realization of $\psi$ would need to be given by an admissible map $F:E\to\PP^1$, which cannot exist since $E$ and $\PP^1$ are not birationally equivalent curves.
\end{exam}

\begin{exam}[Spherical \cco\ groups which are not cyclic]
	$G=\GG_{g,(r,\bar m)}$ is a spherical \cco\ group which is  not cyclic if and only if $g=r=0$ and $\bar m=(m_1,m_2,m_3)$ with $\sum_{j=1}^3\frac{1}{m_j}>1$. In that case, $G$ can be generated by two elements, and there exists an epimorphism $\psi:\mathbb F_2\to G$. Note that the kernel of $\psi$ is a finite index subgroup of $\FF_2$, so it is a finitely generated free group.
	
	Let $U$ be an elliptic curve minus a point, and note that $\pi_1(U)\cong \FF_2$. Suppose that $\psi$ was realized by an admissible
	map $F:U\to\PP^1$. Since algebraic maps are compactifiable, $F$ would extend to an admissible map from a genus $1$ curve to $\PP^1$,
	and, like in the previous example, this is impossible.
\end{exam}

Even if the Orbifold Geometric Realizability Problem~\ref{pbm:orbifoldgeometricmorphism} does not have a positive answer
for any curve orbifold groups $\GG_{g,(r,\bar m)}$ with $\chi_{g,(r,\bar m)}\geq 0$, it is worth noting that it is still
an interesting problem to find restrictions on either $\psi$ or $U$ that could give a positive answer. For example, by
Theorem~\ref{thm:ji-Eva}, the Orbifold Geometric Realizability Problem~\ref{pbm:orbifoldgeometricmorphism} has a solution
for $\ZZ$ and for $\ZZ_2*\ZZ_2$ if $\psi$ is assumed to be an isomorphism (in fact, the proof also works if the
abelianization of the kernel of $\psi$ is finite). In a different direction, the existence of epimorphisms to spherical
\cco\ groups can yield interesting geometric information: let $C\subset\PP^2$ be a sextic curve, and suppose that there
exists an epimorphism $\psi:\pi_1(\PP^2\setminus C)\to D_6$ to the dihedral group of order $6$.
In \cite[Theorem 1.1]{Degtyarev-OkaConjectureII} Degtyarev shows that this implies that $C$ is of torus-type, i.e.,
$C=V(f_2^3+f_3^2)$, where $f_2$ and $f_3$ are homogeneous polynomials of degrees $2$ and $3$ respectively.
Moreover, when $C\subset\PP^2$ has a special type of singularities (including nodes and cusps), then the existence
of certain roots of the Alexander polynomial of $U=\PP^2\setminus C$ determine a positive answer to the Orbifold
Geometric Realization Problem $F:U\to \PP^1_{\bar m}$ such that $\chi_{0,(0,\bar m)}=0$
(see~\cite[Thm. 4.7 and 5.16]{ji-Libgober-mw}).

We end the section by discussing an interesting example where $\chi_{g,(r,\bar m)}>0$ (so $\GG_{g,(r,\bar m)}$ is finite).
\begin{exam}\label{ex:quartic}
		The following is an example of the Orbifold Geometric Realization Problem~\ref{pbm:orbifoldgeometricmorphism}
		for surjections onto non-abelian \emph{spherical} \cco\
		groups with non-trivial finite kernels. This is given by the tricuspidal quartic $C_4\subset\PP^2$.
		Zariski (\cite{Zariski-rational}) used a beautiful argument using the Lefschetz hyperplane section
		Theorem backwards to calculate the fundamental group of the complement of the dual curve to a rational
		nodal curve of degree $d$, since they can be obtained as the plane section of the discriminant variety
		$\Delta_d\subset\PP^3$ of degree $d$ homogeneous polynomials in 2 variables, that is, the configuration
		space of $d$ points in $\PP^1$. Since $C_4$ can be obtained as the dual of a nodal cubic $C_3$, one
		obtains
		$$G=\pi_1(\PP^2\setminus C_4)\cong \pi_1(\PP^3\setminus \Delta_3)\cong \BB_3(\PP^1)
		\cong \langle x,y: xyx=yxy, xy^2x=1\rangle$$
		is the finite metacyclic group of order 12.
		Changing the set of generators $\alpha=xy$, $\beta=xyx$, one can check that
		$G\cong\langle \alpha,\beta: \alpha^2=\beta^3=(\alpha\beta)^2\rangle.$
		Note that $G/\langle \alpha^2\rangle$ is isomorphic to the spherical group $\GG_{0,(2,2,3)}$.
		The existence of the admissible map $F:U=\PP^2\setminus C_4\to \PP^1_{0,(2,2,3)}$ is guaranteed by
		the following functional equation $f_3^2+f_2^3=f_1^2 f_4$ called quasi-toric decomposition
		(see~\cite{ji-Libgober-mw}), where $f_i$ is a homogeneous polynomial in $\CC[x,y,z]$ of degree $i$,
		$C_4=\{f_4=0\}$, and $L=\{f_1=0\}$ is the bitangent line to $C_4$, whose existence can be deduced
		as the dual line of the node of the cubic $C_3$. The conic $V(f_2)$ passes through the three cusps
		and tangent to $L$ at one of the intersections $L\cup C_4$. The map $F=[f_3^2:f_2^3]$ is well
		defined on $U$, $F^*([0:1])$ (resp. $F^*([1:0])$) is a multiple fiber of multiplicity 2 (resp. 3).
		Moreover, $F^*([1:-1])$ is also multiple of multiplicity 2. Hence $F$ induces an orbifold morphism onto
		$\PP^1_{0,(2,2,3)}$ as expected. Maximality is immediate by \cite[Proposition 2.8]{ji-Libgober-mw},
		since the pencil is generated by two multiple curves.

		By construction, there are at least two such pencils depending on the choice of the tangency at
		$L\cup C_4$ resulting in at least two such quasi-toric decompositions. Each decomposition results
		in a different morphism. Note that uniqueness of the realization map $F$ is only guaranteed if
		$\chi_{g,(r,\bar m)}<0$, which is not satisfied in this example, since $\chi_{0,(0,(2,2,3))}=\frac{1}{3}>0$.
	\end{exam}
	
\section{Examples}\label{s:examples}

In this section we discuss some interesting incarnations of Theorem~\ref{thm:GOP} in which $U$ is the
complement of a curve in $\PP^2$.

\begin{exam}[Maximal cuspidal sextic]
	A well-known incarnation of Theorem~\ref{thm:GOP} for $r=1$ and a non-trivial finite kernel
	is given by the maximal cuspidal sextic $C_6\subset\PP^2$, which is dual to the rational nodal quartic.
	Note that $C_6$ has six cusps and four nodes. Using Zariski's argument as in Example~\ref{ex:quartic},
	$$
	\begin{array}{rcl}
		G&=&\pi_1(\PP^2\setminus C_6)\cong \pi_1(\PP^3\setminus \Delta_4)\cong \BB_4(\PP^1)\\
		&\cong& \langle a_1,a_2,a_3: a_ia_{i+1}a_i=a_{i+1}a_ia_{i+1}, i=1,2, a_1a_3=a_3a_1,
		a_1a_2a_3^2a_2a_1=1\rangle.
	\end{array}
	$$
	Changing the set of generators $x=a_1a_2a_1$, $y=a_1a_2$, $z=a_3$ and adding the relation $x^2=1$ one
	can check that
	$G/\langle x^2,zxy\rangle\cong \ZZ_2*\ZZ_3$. The kernel of this map is finite
	(see for instance~\cite[Prop.~2.7]{Margalit-mappingclass}). 
	The existence of the admissible map $F:U=\PP^2\setminus C_6\to \PP^1_{1,(2,3)}$ is guaranteed by the fact
	that $C_6$ must be of torus type, that is $f_6=f_3^2+f_2^2$, where $f_i$ is a homogeneous polynomial in
	$\CC[x,y,z]$ of degree $i$, and $C_6=\{f_6=0\}$. The map $F=[f_3^2:f_2^3]$ is well defined on $U$,
	$F^*([0:1])$ (resp. $F^*([1:0])$) is a multiple fiber of multiplicity 2 (resp. 3). Maximality is immediate
	by \cite[Proposition 2.8]{ji-Libgober-mw}, since the pencil is generated by two multiple curves.
\end{exam}

\begin{exam}[Irreducible quintic $\cC_5$ such that $\pi_1(\PP^2\setminus \cC_5)$ is infinite and non-abelian]\label{exam:quintic}
	As an application of Theorem~\ref{thm:GOP} to projective plane curve complements,
	consider $\cC_5$ an irreducible quintic in $\PP^2$ whose singularities are $\A_6\sqcup 3\A_2$.
	In~\cite{Artal-quintic}, it is shown that
	\begin{equation}
	\label{eq:G237}
	G=\pi_1(\PP^2\setminus \cC_5)=\langle \alpha,\beta,\gamma,v:
	v^5=\alpha^2=\beta^3=\gamma^7=\alpha\beta\gamma,
	[v,\alpha]=[v,\beta]=[v,\gamma]=1\rangle.
	\end{equation}
	Note that $v$ is a central element and $G/\langle v\rangle\cong \GG_{0,(2,3,7)}$.
	Since $\GG_{0,(2,3,7)}$ is a hyperbolic \cco\ group and the kernel of the projection of $G$ onto
	$G/\langle v\rangle$ is cyclic (and thus finitely generated), Theorem~\ref{thm:GOP} implies the
	existence of a pencil inducing an orbifold morphism $\PP^2\setminus \cC_5\to \PP^1_{0,(2,3,7)}$.
	In order to construct this pencil, consider a tricuspidal quartic $\cC_4=\{f_4=0\}$ together with
	its bitangent line $L_1$. Denote by $Q_1$ and $Q_2$ the bitangencies. One can blow up $Q_1$ and
	two of its infinitely near points on the quartic. After that, the self intersection of
	$\tilde L_1$ (the strict transform of $L_1$) is $-1$. One can hence blow down $\tilde L_1$, and
	also $E_2$, and $E_1$ (the second and first exceptional divisors of the blow-ups). Let $\sigma$
	denote the special Cremona transformation obtained by the composition of the three blow-ups
	followed by three blow-downs described above. Note that $\sigma^*(f_4)=g_5g^3_1$, where $g_d$
	is a homogeneous polynomial of degree $d$, $g_1$ is an equation for the strict transform of $E_3$
	and $\cC_5=\{g_5=0\}$ is a quintic with the prescribed singularities. Also, $\sigma^*(l_1)=g^2_1$,
	where $l_1$ is an equation for $L_1$. On the other hand, it is well known that $f_4$ satisfies
	a polynomial identity $f_4l_1^2=f_3^2+f_2^3$, where $\cC_2=\{f_2=0\}$ (resp. $\cC_3=\{f_3=0\}$)
	is a conic (resp. a nodal cubic) passing through all three cusps. One can check that
	$\sigma^*(f_2)=g_4$ (resp. $\sigma^*(f_3)=g_6$). After transforming the polynomial identity
	by $\sigma^*$ one obtains $g_5g_1^7=g_6^2+g_4^3$. Therefore, the map $[x:y:z]\mapsto [g_6^2:g_4^3]$
	is well defined on $\PP^2\setminus \cC_5$ and it induces an orbifold morphism over $\PP^1$ with
	multiplicities $2$ (resp. $3$, resp. $7$) at $[0:1]$ (resp. $[1:0]$, resp. $[1:-1]$).
\end{exam}

\begin{exam}[Rational cuspidal curves \cite{Artal-rational}]\label{exam:Artal}
	Let $d,a,b\in\ZZ$ such that $d>3$, $a\geq b>0$ and $a+b=d-2$. Let $C_{d,a,b}\subset\PP^2$ be a curve of degree $d$ having three singular points $P$, $A$ and $B$ such that the germs $(C_{d,a,b},P)$, $(C_{d,a,b},A)$ and $(C_{d,a,b},B)$ have exactly one Puiseux pair $(d-1,d-2)$, $(2a+1,2)$ and $(2b+1,2)$. Such a curve exists and is unique up to projective equivalence by \cite{FZ}.
	
	Let $n\geq 0$ be such that $2n+1=\gcd(2a+1,2b+1)$. By the main theorem in \cite{Artal-rational} and the presentation given in the proof of \cite[Corollary 1]{Artal-rational},
	$$
	\pi_1(\PP^2\setminus C_{d,a,b})\cong\langle u,v\mid u^2=v^{2n+1}, (v^{-n}u)^{d-2}=v^{d-1}\rangle.
	$$
	Moreover, by \cite[Corollary 1]{Artal-rational}, if $n>0$, $\pi_1(\PP^2\setminus C_{d,a,b})$ is a central extension of the \cco\ group $\GG_{0,(2,2n+1,d-2)}$ (which is a triangle group) as follows:
	\begin{equation}\label{eq:central}
		\begin{array}{ccccrclcc}
			1 &\to & K &\to &\pi_1(\PP^2\setminus C_{d,a,b})&\to &\GG_{0,(2,2n+1,d-2)}&\to& 1\\
			&&&& u&\mapsto  & x_1\\
			&&&& v&\mapsto  & x_2^2\\
		\end{array}
	\end{equation}
 where $K$ is the cyclic group generated by $u^2$, and $\GG_{0,(2,2n+1,d-2)}$ has the presentation given in equation~\ref{eq:Ggm}. 
 
 Note that, if $n>0$ and $(d,a,b)\neq (4,1,1), (7,4,1)$, then $\chi_{0,(2,2n+1,d-2)}<0$. In that case, the
 epimorphism from \eqref{eq:central} provides infinitely many examples satisfying the hypotheses of
 Theorem~\ref{thm:GOP}. Note that the geometric orbifold morphisms onto $\PP^1$ obtained by applying Theorem~\ref{thm:GOP} to these examples were already
 described in~\cite[Prop. 2.13]{ACM-multiple-fibers}.
\end{exam}

\begin{exam}\label{ex:Oka2}
	An example of Theorem~\ref{thm:GOP} for $r>0$ and a non-trivial kernel can be found in~\cite[Cor.~1.4]{Oka-genericRjoin}. Here, Oka computes the fundamental group $G$ of the complement of
	a projective curve $C=\cup_{k=0}^{r} C_{[1:\omega^k]}$ which is a union of
	$r$ curves of type $C_{[1:\omega^k]}=\{\omega^k f- g=0\}\subset\PP^2$, where
	\begin{itemize}
		\item $\omega$ is a primitive $r$-th root of unity,
		\item $d=\deg(f)=\deg(g)$,
		\item $f=\prod_{j=1}^{l_1} (X-\beta_jZ)^{m_{1j}}$ and $g=\prod_{i=1}^{l_2} (Y-\alpha_iZ)^{m_{2i}}$, where $\alpha_1,\ldots,\alpha_{l_2},\beta_1,\ldots,\beta_{l_1}$ are mutually distinct complex numbers,
		\item $m_1=\gcd(m_{1j})$ and $m_2=\gcd(m_{2i})$ are coprime, and
		\item the singularities of $C$ are contained in $V(f)\cap V(g)$.
	\end{itemize}
	Then, Oka proves that $G$ is a finite central extension of an open curve orbifold group as follows
	$$
	1\to \ZZ_n \to G \to \FF_{r-1} * \ZZ_{m_1} * \ZZ_{m_2}\to 1,
	$$
	where $n=\frac{d}{m_1m_2}$.
	
	The existence of the
	admissible map $$F:\PP^2 \setminus C \to \PP^1\setminus \cup_k\{[1:\omega^k]\}$$
	realizing the epimorphism $ G \to \FF_{r-1} * \ZZ_{m_1} * \ZZ_{m_2}$ is guaranteed by the construction of $C$. That is, the pencil map
	$F=[f:g]$ is admissible, and it induces an orbifold morphism $F:\PP^2 \setminus C \to \PP^1_{r,(m_1,m_2)}$ because the preimage
	of $[0:1]$ (resp. $[1:0]$) is a fiber of multiplicity $m_1$ (resp. $m_2$). If the orbifold Euler characteristic of $\FF_{r-1} * \ZZ_{m_1} * \ZZ_{m_2}$ is negative, maximality follows by Lemmas~\ref{lem:NINF} and~\ref{lem:FgroupNINF}, and in particular the multiple fibers of $F$ are contained in $\{[0:1],[1:0]\}$.
\end{exam}

\section{Serre's question in $\PP^2$ for curve orbifold groups}\label{s:planecurves}

In this section we address Problem~\ref{prob:serre} for curve orbifold groups, that is, the question of what curve orbifold groups that can be realized as
fundamental groups of curve complements in~$\PP^2$.

\begin{thm}\label{thm:realizable-spherical-euclidean}
	The only spherical and Euclidean \cco\ groups that can be realized as the complement of a curve in $\PP^2$
	are $\ZZ^2$ and the finite cyclic groups.
\end{thm}

\begin{proof}
	Note that $\ZZ^2\cong\pi_1((\CC^*)^2)$, and $(\CC^*)^2$ is the complement in $\PP^2$ of three non-concurrent lines.
	Also, the complement of a smooth curve of degree $d$ in $\PP^2$ is cyclic of order $d$. Hence, all that is left to
	prove is that the other groups in List~\ref{list} cannot be realized as the complement of a curve in $\PP^2$.
	
	Let us compute the abelianization of these groups:
	$$
	\begin{array}{|c|c|}
		\hline
		\text{Groups}&\text{Abelianization}\\\hline
		A_5=\GG_{0,(2,3,5)} & \{0\}\\
		S_4=\GG_{0,(2,3,4)}, \GG_{0,(2,2,2k+1)}\,\forall k\geq 1 & \ZZ_2\\
		A_4=\GG_{0,(2,3,3)} & \ZZ_3\\
		\GG_{0,(2,3,6)} & \ZZ_6\\
		\GG_{0,(2,2,2k)}\,\forall k\geq 1 &\ZZ_2\times\ZZ_2\\
		\GG_{0,(3,3,3)} &\ZZ_3\times\ZZ_3\\
		\GG_{0,(2,2,2,2)} &\ZZ_2\times\ZZ_2\times \ZZ_2\\
		\GG_{0,(2,4,4)} &\ZZ_2\times\ZZ_4\\
		\hline
	\end{array}
	$$
	The fundamental group of a curve complement in $\PP^2$ has a finite abelianization if and only
	if the curve is irreducible, in which case the abelianization is cyclic and its order coincides
	with the degree of the curve. This rules out the last four cases in the table above. If the
	abelianization is $\ZZ_2$ (resp. $\ZZ_3$), the curve would need to be an irreducible conic
	(resp. cubic). The fundamental of the complement of irreducible curves of degrees less than four
	is abelian, which rules out all the remaining cases except for~$\GG_{0,(2,3,6)}$.
	
Suppose that there exists a curve $C$ in $\PP^2$ such that $\GG_{0,(2,3,6)}$. Then, $C$ must be an
irreducible sextic. Note that $\GG_{0,(2,3,6)}\cong\langle x,y\mid x^2=y^3=(xy)^6=1\rangle$ surjects
onto the dihedral group of order $6$, sending $x$ to a reflection and $y$ to a rotation.
By~\cite[Theorem 1.1]{Degtyarev-OkaConjectureII} $C$ is of torus-type, i.e., $C=V(f_2^3+f_3^2)$,
where $f_2$ and $f_3$ are homogeneous polynomials of degrees $2$ and $3$ respectively. Since the
generic fiber of the map $F=[f_2^3:f_3^2]:\PP^2\setminus C\to\PP^1\setminus\{[1:-1]\}$ is irreducible,
it induces an epimorphism (see~\cite[Prop. 1.4]{ACM-multiple-fibers})
$F_*:\pi_1(\PP^2\setminus C)\to \pi_1^{\orb}(\PP^1_{1,(2,3)})$, where $\PP^1\setminus\{[1:-1]\}$ is
endowed with the maximal orbifold structure with respect to $F$, in which the points of multiplicity
greater than $1$ are $[0:1]$ (with multiplicity 3) and $[1:0]$ (with multiplicity $2$). Hence,
$$
F_*:\GG_{0,(2,3,6)}\twoheadrightarrow \ZZ_2*\ZZ_3,
$$
and this is impossible, as no such epimorphism exists. Indeed, $\ZZ_2*\ZZ_3$ surjects onto
$\GG_{0,(2,3,6)}$, and $\ZZ_2*\ZZ_3$ is Hopfian, but $\GG_{0,(2,3,6)}$ is not isomorphic
to $\ZZ_2*\ZZ_3$ (for example the commutator group of the first one is isomorphic to $\ZZ^2$, and the
commutator group of the second one is free).
\end{proof}

The following result rules out infinitely many hyperbolic \cco\ groups 
from being the fundamental group of a plane curve complement.

\begin{lemma}\label{lem:pi1CCO}
	Let $\mathbb G_{g,\bar m}$ be a hyperbolic \cco\ group, where $\bar m=(m_1,\ldots,m_n)$, $m_i\geq 2$ for all $i=1,\ldots,n$, and $n,g\geq 0$. 
	Suppose that there exists a plane curve $C\subset\PP^2$ such that $\pi_1(\PP^2\setminus C)\cong \GG_{g,\bar m}$. Then,
	\begin{enumerate}
		\item\label{pi1CCO1} $g=0$ and $n=3$.
		\item\label{pi1CCO2} $C$ is irreducible of degree $d=\gcd(m_1m_2,m_2m_3,m_3m_1)$.
		\item\label{pi1CCO3}\label{pi1CCO4} Up to reordering of the entries of $\bar m$, we may assume that
		$$
		\mathbb G_{g,\bar m}=\langle x,y,z\mid x^{m_1}=y^{m_2}=z^{m_3}=xyz=1\rangle,
		$$
		where $\gcd(m_1,m_2)=1$ and moreover 
		$\mathbb G_{g,\bar m}$ can be generated by $d-2$ conjugates of $z^k$ for some $k\geq 1$ such that $\gcd(k,m_3)=1$ .
	\end{enumerate} 
\end{lemma}
\begin{proof}
	By Theorem~\ref{thm:GOP}, there
	exists an admissible map $F:\PP^2\setminus C\to D$ onto a compact curve $D$ of genus $g$
	such that $F_*:\pi_1(\PP^2\setminus C)\to \pi_1^{\orb}\left(D_{\bar m}\right)$ is an isomorphism, and
	the orbifold structure $\bar m$ is maximal. Since $F$ defines a surjective rational map $F:\PP^2\dashrightarrow D$, then $g=0$ and $D_{\bar m}=\PP^1_{\bar m}$. Note that $H_1(\PP^2\setminus C,\ZZ)$ is a
	finite group, and thus $C$ is irreducible (see part~\eqref{pi1CCO2}) and $H_1(\PP^2\setminus C,\ZZ)\cong\ZZ_d$, where $d$ the
	degree of $C$ \cite[Prop 1.3, Chapter 4]{Dimca-singularities}. Also note that $\chi_{0,\bar m}<0$ implies
	$\bar m=(m_1,\ldots, m_n)$ is such that~$n\geq 3$.
	
	We may extend $F$ to its maximal domain of definition in $\PP^2$, which will be the complement of a finite number of points $\mathcal B$. By \cite[Proposition 2.8]{ji-Libgober-mw}, 
	$F:\PP^2\setminus\mathcal B\to \PP^1$ has at most $2$ multiple fibers. This is only
	possible if $n=3$ and up to reordering of the $m_j$'s, $F(C\setminus\mathcal B)$ is the point
	$P_3\in\PP^1$ of multiplicity $m_3$. In particular, this concludes the proof of part~\eqref{pi1CCO1}. 
	Note also that since $F$ is admissible, the multiplicities $m_1$
	and $m_2$ must be coprime, and thus the first statement of part~\eqref{pi1CCO3} follows.
	
	Since $m_1$ and $m_2$ are coprime, the abelianization of $\GG_{0,(m_1,m_2,m_3)}$ is the cyclic
	group of order $\gcd(m_1m_2,m_3)=\gcd(m_1m_2,m_2m_3,m_3m_1)$, which coincides with the degree $d$ of $C$. This concludes the proof of part~\eqref{pi1CCO2}
	
	Up to change of coordinates in $\PP^1$, $F$ is of the form $[f_{lm_1}^{m_2}:f_{lm_2}^{m_1}]$, where $l\geq 1$ and $f_{lm_i}$ is homogeneous of degree $lm_i$ in $3$ variables for $i=1,2$. Furthermore, we may assume that $C$ lies in the fiber over $[1:-1]$. Let $g$ be a (reduced) defining polynomial of $C$. Then, there exists a non-constant homogeneous polynomial $h$ such that the following quasi-toric relation holds for some $k\geq 1$ such that $\gcd(k,m_3)=1$:
	$$
	f_{lm_1}^{m_2}+f_{lm_2}^{m_1}=h^{m_3}g^k.
	$$
	The conjugacy class in $\pi_1(\PP^2\setminus C)$ of the class of any meridian around $C$ consists of all the classes of  meridians around $C$ with the same orientation \cite[Prop. 1.34]{ji-pau}. Using the isomorphism $$F_*:\pi_1(\PP^2\setminus C)\to\pi_1^{\orb}(\PP^1_{(m_1,m_2,m_3)})\cong\GG_{0,(m_1,m_2,m_3)},$$
	one obtains that $z^k$ is the image of a meridian of $C$.
	
	Note that if $d=1,2$, $\pi_1(\PP^2\setminus C)$ is  finite (and abelian) and thus not a hyperbolic \cco\ group. Hence, part~\eqref{pi1CCO3} is a consequence of the following statement: If $C$ is an irreducible curve of degree $d\geq 3$, then $\pi_1(\PP^2\setminus C)$ can be generated by $d-2$ or less conjugates of $\gamma$, where $\gamma$ is any meridian around $C$. Let us prove the statement. If $C$ is smooth, then $\pi_1(\PP^2\setminus C)$ is abelian (\cite[Cor 3.18, Chapter 4]{Dimca-singularities}) and thus $\pi_1(\PP^2\setminus C)$ is finite cyclic. In particular, $\pi_1(\PP^2\setminus C)$ is generated by one meridian around $C$. If $C$ is not smooth, let $P$ be a singularity of $C$, and after a change of coordinates in $\PP^2$ we may assume that $P=[1:0:0]$. Consider the projection $\PP^2\setminus\{P\}\to\PP^1$ given by $[a:b:c]\to[b:c]$. Note that the closure of every fiber of this projection is a line in $\PP^2$ going through the singular point $P$ which is transversal to $C$ at every  intersection point other than $P$, and these other intersection points are smooth points of $C$. Hence, the generic fibers of this projection are isomorphic to $\PP^1$ minus at most $d-1$ points (namely $P$ and at most $d-2$ smooth points of $C$), and the fundamental group of any of these generic fibers is generated by at most $d-2$ meridians around $C$. Since the map $\PP^2\setminus C\to\PP^1$ is admissible, with no multiple fibers, and $\PP^1$ is simply connected, the result follows from Lemma~\ref{lem:exact}.
\end{proof}

The following result, joint with Lemma~\ref{lem:pi1CCO}, implies that no irreducible curve of degree $d\leq 6$ in $\PP^2$ can be such that the fundamental group of its complement is isomorphic to a hyperbolic \cco\ group. Note that these fundamental groups are classified up to degree $d=5$, but not in degree $6$ or higher. The proof in the case $d=6$ uses Lemma~\ref{lem:someEvenDegree}, which will be proved later, and the characterization of sextics of torus type due to Degtyarev in \cite[Theorem 1.1]{Degtyarev-OkaConjectureII}.

\begin{thm}\label{thm:upToSextic}
	Let $\GG_{0,(m_1,m_2,m_3)}$ be a hyperbolic triangle group, where $$d=\gcd(m_1m_2,m_2m_3,m_3m_1)\leq 6.$$ Then, $\GG_{0,(m_1,m_2,m_3)}$ is not isomorphic to $\pi_1(\PP^2\setminus C)$ for any plane curve $C$.
\end{thm}

\begin{proof}
	We argue by contradiction. Suppose that there exists a curve $C\subset\PP^2$ such that $\pi_1(\PP^2\setminus C)\cong \GG_{0,(m_1,m_2,m_3)}$. By Lemma~\ref{lem:pi1CCO}\eqref{pi1CCO2}, $C$ is irreducible of degree $d$. 
	
	Fundamental groups of complements of irreducible projective plane curves of degrees up to 5 are known to
	be finite with the only exception of the quintic shown in Example~\ref{exam:quintic}
	(see~\cite{Degtyarev-quintic} for the calculation of fundamental groups for degree 5, the other
	cases are well known in the literature), which is not a hyperbolic triangle group: if it was,
	it would satisfy property \NINF\ by Lemma~\ref{lem:FgroupNINF}. This contradicts the fact that
	it has a non-trivial finitely generated normal subgroup of infinite index, namely the central
	subgroup generated by the central element $v$ in the presentation~\eqref{eq:G237}. Hence, we reach a contradiction if $d\leq 5$.
	
	Suppose that $d=6$. By Lemma~\ref{lem:pi1CCO}\eqref{pi1CCO3}, we may reorder the entries of $(m_1,m_2,m_3)$ and assume that $\gcd(m_1,m_2)=1$, so $6\vert m_3$. Hence, up to reordering, $(m_1,m_2,m_3)$ is as in one of these two cases:
	\begin{itemize}
		\item Case 1: $(m_1,m_2,m_3)=(2a,3b,6c)$, where $\gcd(2a,3b)=1=\gcd(ab,c)$ and $a,b,c\geq 1$. Moreover, at least one of $a,b,c$ is $\geq 2$ (or else $\GG_{0,(m_1,m_2,m_3)}$ would not be a hyperbolic triangle group).
		\item Case 2: $(m_1,m_2,m_3)=(6k\pm 1,6b,6c)$, where $\gcd(6k\pm 1,6b)=1=\gcd((6k\pm 1)\cdot b,c)$ and $k,b,c\geq 1$.
	\end{itemize}
	
	Lemma~\ref{lem:someEvenDegree} implies a contradiction in case 2. Finally, assume that case 1 holds. Consider the presentation
	$$
	\GG_{0,(2a,3b,6c)}=\langle x, y, z\mid x^{2a}=y^{3b}=z^{6c}=xyz=1\rangle
	$$
	and the presentation of the dihedral group $D_6$ of order $6$
	$$
	D_6=\langle u, v\mid u^2=v^3=1, uvu=v^{-1}\rangle.
	$$
	Consider the epimorphism $\GG_{0,(2a,3b,6c)}\twoheadrightarrow D_6$ such that $x\mapsto u$, $y\mapsto v$ and $z\mapsto (uv)^{-1}$. By \cite[Theorem 1.1]{Degtyarev-OkaConjectureII}, the existence of this epimorphism implies that the sextic $C$ is of torus type, i.e. $C=V(f_2^3+f_3^2)$ for some homogeneous polynomials $f_2,f_3\in\CC[x,y,z]$ of degrees $2$ and $3$ respectively. In particular the (necessarily admissible) map $[f_2^3:f_3^2]:\PP^2\setminus C\to \PP^1\setminus \{[-1:1]\}$ induces an epimorphism
	$$
	\GG_{0,(2a,3b,6c)}\cong\pi_1(\PP^2\setminus C)\twoheadrightarrow \ZZ_2*\ZZ_3
	$$
	with finitely generated kernel at the level of (orbifold) fundamental groups (see Lemma~\ref{lem:exact}). Since $\GG_{0,(2a,3b,6c)}$ and $\ZZ_2*\ZZ_3$ both satisfy property~\NINF\ by Lemma~\ref{lem:FgroupNINF}, Lemma~\ref{lem:NINF} implies that $\GG_{0,(2a,3b,6c)}\cong \ZZ_2*\ZZ_3$, which contradicts Theorem~\ref{thm:CCOnotFreeProduct}.
\end{proof}

The proof of Theorem~\ref{thm:upToSextic} ends after proving the following.

\begin{lemma}\label{lem:someEvenDegree}
Let $\GG_{0,(m_1,m_2,m_3)}$ be a hyperbolic triangle group, and that, up to reordering,  $(m_1,m_2,m_3)=(a,2db,2dc)$ where $a=2dk\pm 1$, $k,b,c\geq 1$, $d\geq 3$ and $\gcd(a,b)=\gcd(b,c)=\gcd(c,a)=1$. Then, $\GG_{0,(m_1,m_2,m_3)}$ is not isomorphic to $\pi_1(\PP^2\setminus C)$ for any plane curve $C$.
\end{lemma}
\begin{proof}
Let us argue by contradiction. Suppose that there exists a curve $C$ such that $\pi_1(\PP^2\setminus C)$ is isomorphic to such triangle group. By Lemma~\ref{lem:pi1CCO}\eqref{pi1CCO2}, $C$ is irreducible of degree $2d$. By Lemma~\ref{lem:pi1CCO}\eqref{pi1CCO3}, we may reorder $b$ and $c$ to find the following presentation of $\pi_1(\PP^2\setminus C)$:
$$
\langle x,y,z\mid x^{2dk\pm 1}=y^{2db}=z^{2dc}=xyz=1\rangle,
$$
where $z^s$ is a meridian around $C$ for some $s\geq 1$ such that $\gcd(s,2dc)=1$. In particular, $s$ is odd.

In each of the following cases, consider the following group homomorphism $\varphi$ to the symmetric group $S_n$ on $n$ elements:
\begin{itemize}
	\item Case 1: $a=2dk-1$. Then, $n=2dk$ and
	$$
\begin{array}{cccl}
	\varphi: & \mathbb G_{0,(2dk-1,2db,2dc)}&\longrightarrow & S_{2dk}\\
	& x &\longmapsto &\underbrace{(1,2,\ldots,2dk-1)}_{(2dk-1)-\text{cycle}}\\
	& y &\longmapsto & \left(1,2dk,2dk-1,\ldots,2d(k-1)+2\right)\cdot\\
	&&&\left(2,2d(k-1)+1,2d(k-1),\ldots,2d(k-2)+3\right)\cdot\\
	&&&\ldots\cdot\left(k,k+2d-1,k+2d-2,\ldots,k+1\right)\\
	& z &\longmapsto & \left(1,2dk\right)\left(2,2d(k-1)+2\right)\left(3,2d(k-2)+3\right)\ldots\left(k,k+2d\right),
\end{array}
$$
	where the image of $y$ is the product of $k$ disjoint $2d$-cycles, and the image of $z$ is the product of a $k$ disjoint transpositions.
	\item Case 2: $a=2dk+1$. Then, $n=a=2dk+1$, and
	$$
\begin{array}{cccl}
	\varphi: & \mathbb G_{0,(2dk+1,2db,2dc)}&\longrightarrow & S_{2dk+1}\\
	& x &\longmapsto &\underbrace{(1,2,\ldots,2dk+1)}_{(2dk+1)-\text{cycle}}\\
	& y &\longmapsto &  (2,2dk+1,2dk,\ldots,2d(k-1)+3)\cdot\\
	&&& (3,2d(k-1)+2,2d(k-1)+1,\ldots,2d(k-2)+4)\cdot\\
	&&& \ldots\cdot(k+1,k+2d,k+2d-1,\ldots,k+2)\\
	& z &\longmapsto & (1,2)(3,2d(k-1)+3)(4,2d(k-2)+4)\ldots(k+1,k+2d+1),
\end{array}
$$
	where the image of $y$ is the product of $k$ disjoint $2d$-cycles and the image of $z$ is the product of $k$ disjoint transpositions.
\end{itemize}

By Lemma~\ref{lem:pi1CCO}\eqref{pi1CCO3} there exist $\gamma_i\in \mathbb G_{(2dk\pm 1,2db,2dc)}$ conjugate to $z^s$ for $i=1,\ldots, 2d-2$ such that $\gamma_1,\ldots,\gamma_{2d-2}$ generate $\mathbb G_{(2dk\pm 1,2db,2dc)}$. Then, $\varphi(\gamma_1),\ldots,\varphi(\gamma_{2d-2})$ generate $\im(\varphi)$, which is a transitive subgroup of $S_n$.	

Consider the graph with $n$ vertices labeled $1,\ldots,n$, and such that vertices $j$ and $l$ are joined by an edge if $\varphi(\gamma_i)(j)=l$ for some $i=1,\ldots, 2d-2$. Since $\varphi(\gamma_1),\ldots,\varphi(\gamma_{2d-2})$ generate the transitive subgroup $\im(\varphi)$, this graph must be connected. However, since $s$ is odd, $\varphi(\gamma_i)$ contributes the same number of edges to the graph as $\varphi(z)$ would, that is, $k$ edges. Hence, in case 1 (resp. in case 2) the graph has $n=2dk$ (resp. $n=2dk+1$) vertices and $(2d-2)k$  edges, so this graph cannot be connected and we reach a contradiction.
\end{proof}

\begin{cor}\label{cor:Serre}
	The table in Figure~\ref{fig:serre} summarizes the current state of Serre's problem~\ref{prob:serre} for curve orbifold groups.
\end{cor}

\begin{proof}
This problem has a positive answer for open curve orbifold groups if and only if there are at most two orbifold points of coprime multiplicity, that is, 
$\textrm{length}(\bar m)\leq 2$ and $\gcd(m_1,m_2)=1$ (the top-left cell in Figure~\ref{fig:serre}) in Figure. Realizations for this case can be found in~\cite[Thm. 1.2]{ji-Eva-orbifold} and Example~\ref{ex:Oka2}. A proof that this condition is necessary (the top-right cell  in Figure~\ref{fig:serre}) can be found in~\cite[Cor. 3.14]{ji-Eva-orbifold}.
The other positive answer comes for compact curve orbifold groups when either $g=0$ and $\textrm{length}( \bar m)\leq 2$ (finite cyclic groups)
or $g=1$ and no orbifold points ($\ZZ^2$) and it is shown in Theorem~\ref{thm:realizable-spherical-euclidean}.
Negative answers to the problem are shown in Lemma~\ref{lem:pi1CCO}, Theorem~\ref{thm:upToSextic} and Lemma~\ref{lem:someEvenDegree}.
The problem is open for the remaining groups, namely, the leftover hyperbolic triangle groups.
\end{proof}

\bibliographystyle{amsplain}
\bibliography{eva-ji}

\end{document}